\documentclass[a4paper,fleqn]{cas-dc}

\usepackage{amsmath,amssymb,amsfonts}
\usepackage{ntheorem}
\pdfoutput=1
\usepackage{algorithmic}
\usepackage{graphicx}
\usepackage{textcomp}
\usepackage{chemarrow}
\usepackage{subfigure}
\usepackage[numbers]{natbib}

\def\tsc#1{\csdef{#1}{\textsc{\lowercase{#1}}\xspace}}
\tsc{WGM}
\tsc{QE}

\begin{document}
\let\WriteBookmarks\relax
\def\floatpagepagefraction{1}
\def\textpagefraction{.001}

\shorttitle{Chemical relaxation oscillator designed to control molecular computation}    

\shortauthors{X. Shi and C. Gao}  

\title [mode = title]{Design of universal chemical relaxation oscillator to control molecular computation}  

\tnotemark[1] 

\tnotetext[1]{This work was funded by the National Nature Science Foundation of China under Grant No. 12071428 and 62111530247, and the Zhejiang Provincial Natural Science Foundation of China under Grant No. LZ20A010002.} 

%

\author[1]{Xiaopeng Shi}[style=Chinese]


\ead{12035033@zju.edu.cn}



\affiliation[1]{organization={School of Mathematical Science},
            addressline={Zhejiang University}, 
            city={Hangzhou},
            postcode={310000}, 
            state={},
            country={PR China}}

\author[1]{Chuanhou Gao}[style=Chinese]

\cormark[2]

\ead{gaochou@zju.edu.cn}


\credit{}


\cortext[2]{Corresponding author}



\begin{abstract}
Embedding efficient command operation into biochemical system has always been a research focus in synthetic biology. One of the key problems is how to sequence the chemical reactions that act as units of computation.
The answer is to design chemical oscillator, a component that acts as a clock signal to turn corresponding reaction on or off.
Some previous work mentioned the use of chemical oscillations. However, the models used either lack a systematic analysis of the mechanism and properties of oscillation, or are too complex to be tackled with in practice. 
Our work summarizes the universal process for designing chemical oscillators, including generating robust oscillatory species, constructing clock signals from these species, and setting up termination component to eventually end the loop of whole reaction modules. 
We analyze the dynamic properties of the proposed oscillator model in the context of ordinary differential equations, and discuss how to determine parameters for the effect we want in detail.
Our model corresponds to abstract chemical reactions based on mass-action kinetics which are expected to be implemented into chemistry with the help of DNA strand displacement cascades. Our consideration of ordering chemical reaction modules helps advance the embedding of more complex calculations into biochemical environments.
\end{abstract}


\begin{highlights}
\item We give an approach of designing universal chemical oscillator for synchronous sequential computation in biochemical systems.
\item We analyse the properties of the oscillator model in detail, such as period and amplitude, as well as the dynamic behaviour under different parameter values.
\item We test the adjustment performance of our chemical oscillator based on the counting model.
\item We give a method for the system to terminate the loop spontaneously.
\end{highlights}


\maketitle

\section{Introduction}
\label{sec:introduction}
A main desire of synthetic biology is designing programmable chemical controller which can operate in molecular contexts incompatible with traditional electronics \cite{vasic2020}. We have learned plenty of algorithms from how life works such as artificial neural network and genetic algorithm, while on the contrary, inserting advanced computational methods into living organisms to accomplish specific tasks e.g. biochemical sensing and drug delivery is also fascinating. A great deal of related work has sprung up in recent years: Moorman et al. proposed a biomolecular perceptron network in order to recognize patterns or classify cells \emph{in-vivo} \cite{moorman2019dynamical}. The beautiful work of Vasic et al. put the feed-forward RRelu neural network into chemical reaction networks(CRNs), and performed their model on standard machine learning training sets \cite{vasic2021programming}. There were also some attempts to build CRNs capable of learning \cite{chiang2015reconfigurable,blount2017feedforward}. However, no one has implemented the whole neural network computation(including feed-forward and parameter-learning process) into biochemical system. The main reason is that algorithm based on computer instruction performs operations in a sequential manner whereas biochemical reactions proceed synchronously. This contradiction calls for an appropriate mediation method, which isolates two reaction modules \cite{vasic2020} from co-occurring and controls their order. D Blount et al. constructed cell-like compartments and added clock signals artificially in order to solve this problem, which increased the difficulty of biochemical implementation \cite{blount2017feedforward}. A more natural idea is to design chemical oscillators which produce periodical clock signals automatically, taking advantage of their phase change to turn corresponding reaction module on or off.
\par Oscillation phenomena are often encountered in chemical and biological systems such as Belousov-Zhabotinskii reaction \cite{tyson2013} and circadian rhythm \cite{forger2017biological}, while dealing with reaction orders by chemical oscillators is not a groundless rumour. Arredondo and Lakin \cite{arredondo2022supervised} utilized a 20-dimensional oscillator extended from ant colony model \cite{lachmann1995computationally} to order the parts of their chemical neural networks. Jiang et al. introduced a different oscillator model with 12 species and 24 reactions \cite{jiang2011}, then chose two of these species to serve as clock signals. Their work follows the same logic: Firstly, find a suitable oscillator model and give a set of appropriate parameters(along with initial values), then confirm that the model is indeed available for use by simulation. There are two main problems with such design, one is the lack of theory about oscillation mechanism i.e. it is often unclear why these models produce oscillatory behaviour. The other one is that these oscillators belong to harmonic type, whose amplitude and period may not recover the initial values after a perturbation. In order to obtain a satisfactory oscillator structure, these models are highly required to select accurate initial values, which causes difficulties during biochemical implementation. In view of this, we consider to design a set of oscillator models based on transparent mechanism, making sure why the oscillation behavior occurs and how it evolves are clear, and selection of initial values is robust. We also give the relationship between period and parameters in our oscillator model. 
\par Our ultimate goal is to perform our chemical oscillator \textit{in-vivo} together with other operational modules based on chemical reactions, so concentrations of chemical species play the role in our oscillator model. The whole process consists of three steps: construct oscillator model in the context of dynamical system first, then select the appropriate kinetics(mainly mass action kinetics) to put it back into abstract chemical reaction networks \cite{feinberg2019foundations}, and finally utilize DNA strand displacement cascades \cite{soloveichik2010dna} to implement them into chemistry. Since each step of the above transformation process has a relatively mature theory as a guarantee, it is reasonable to carry out our work on the theoretical level of ODE and dynamical system.
\par We focus on designing chemical oscillators for the sequence of two chemical reaction modules and making sure that our method is still valid when faced with the task of ordering multiple reaction modules. Both controlling the sequence and alternating cycles of two reaction modules are very common in molecular operations and synthetic biology, such as module instructions that involve judgment before execution, or reaction modules that realize the loop of feed-forward transmission and back propagation in artificial neural networks. Not only do we provide a common approach of designing transparent oscillator for such requirements, but we also offer a method for how to let the modules terminate the loop according to the judgment statement spontaneously.

\par This paper is organized as follows. Related definitions are given in section \uppercase\expandafter{\romannumeral2}. Section \uppercase\expandafter{\romannumeral3} exhibits the structure of 4-dimensional universal oscillator model based on 2-dimensional relaxation oscillation, which is able to generate a pair of desired oscillatory component with proper selection of parameters. In section \uppercase\expandafter{\romannumeral4} the dynamical behaviours involved in this model are analyzed in detail, and the amplitude and period of oscillatory components are estimated with appropriate parameter values.
 Then we talk about ways to make the system spontaneously terminate the loop in section \uppercase\expandafter{\romannumeral5}. We summarize the general process of placing our oscillator components into reaction modules to be ordered with example of chemical neural network in section \uppercase\expandafter{\romannumeral6}. And finally, section \uppercase\expandafter{\romannumeral7} is dedicated to conclusion and discussion of the whole paper.

\section{Related definitions}
In this section we provide the preparatory knowledge such as definition and concept of chemical reaction network first, following the work of Feinberg \cite{feinberg2019foundations} and Anderson et al \cite{anderson2021reaction}. Then based on our example of our reaction modules, we talk about the design requirements for chemical oscillators.
\subsection{fundamental concept of chemical reaction network(CRN)}
\newtheorem{definition}{Definition}[section]
\begin{definition}
    A \textit{chemical reaction network(CRN for short)} consists of nonempty and finite set of species $ 
\mathcal{S}$ and finite set of complexes $\mathcal{C}$ and set of reactions $\mathcal{R}$ satisfying the following description:
\begin{itemize}
    \item[*] Elements of species $\mathcal{S}$ act as fundamental components in CRN.
    \item[*] Every complex in $\mathcal{C}$ is a linear combination of species over the non-negative integers.
    \item[*] Two complexes connected by arrow form a reaction belonging to $\mathcal{R}$.
    \item[*] Species to the left of the arrow in each reaction are called reactants for that reaction, and the species to the right are called products.
\end{itemize}
\end{definition}

\par We often denote the species set as $\mathcal{S}=\left\{X_{1},...,X_{n}\right\}$, in which case the complexes are of the form $a_{1}X_{1}+\cdots +a_{n}X_{n}$, where $a_{i}\in \mathbb{Z}_{\geqslant 0}$ for each $i \in \left\{1,...,n\right \}$. Then the reaction set $\mathcal{R}=\left \{R_{1},...R_{m} \right \}$ and $R_{i}$ is just like
\begin{equation*}
a_{i1}X_{1}+\cdots +a_{in}X_{n} \overset{k_{i}}{\rightarrow} b_{i1}X_{1}+\cdots +b_{in}X_{n}\ ,
\end{equation*}
for $k_{i}$ is the rate constant of this reaction.
Based on different kinetic assumptions, we can model ordinary differential equations(ODEs for short) for species concentration changes according to a given chemical reaction network. This paper chooses the most common form of kinetics termed \textit{mass-action kinetics}:
\begin{equation*}
    \dot{x} = \Gamma \cdot v\left(x\right).
\end{equation*}
In which $x \in \mathbb{R}^{n}_{\geqslant 0}$ represents the concentration of species $X_{1},...,X_{n}$, coefficient matrix $\Gamma_{n\times m}$ satisfies $\Gamma_{ij} = b_{ij}-a_{ij}$ and rate function $v\left (x \right ) = \left ( k_{1}\prod_{i=1}^{n}x_{i}^{a_{i1}},...,k_{m}\prod_{i=1}^{n}x_{i}^{a_{im}}  \right )^{\top}$.
~\\
\newtheorem{example}{Example}[section]
\begin{example}
Consider the following reaction system:
\begin{align*}
    2X_{1} &\overset{k_{1}}{\rightarrow} X_{2} + X_{3}\ ,\\
    X_{3} &\overset{k_{2}}{\rightarrow} 2X_{1}\ ,
\end{align*}
with the species set  $\mathcal{S}=\left \{ X_{1}, X_{2}, X_{3} \right \}$, the complex set $\mathcal{C}=\left \{2X_{1}, X_{2}+X_{3},X_{3} \right \}$, the stoichiometric matrix $\Gamma_{n\times m} = \begin{pmatrix}
-2&2\\ 
 1&0 \\ 
 1&-1 
\end{pmatrix}$ and the vector-valued rate function $v\left (x \right ) = \left ( k_{1}x_{1}^{2}, k_{2}x_{3} \right )^{\top}$. The ODEs are 
\begin{align*}
    \frac{\mathrm{d} x_{1}}{\mathrm{d} t} &= -2k_{1}x_{1}^{2} + 2k_{2}x_{3}\ , \\
\frac{\mathrm{d} x_{2}}{\mathrm{d} t} &= k_{1}x_{1}^{2}\ ,  \\
\frac{\mathrm{d} x_{3}}{\mathrm{d} t} &= k_{1}x_{1}^{2} - k_{2}x_{3}\ .
\end{align*}
\end{example}

Based on expression of ODEs, we give the definition of catalyst, which would be the form of the oscillatory component that we construct as clock signal participating in the reaction modules.
~\\
\begin{definition}
    We call a species $X_{i}$ catalyst of a specific CRN for that $\frac{\mathrm{d} x_{i}}{\mathrm{d} t} = 0$.
\end{definition}

\subsection{Reaction Modules in Molecular Calculation}
 The concept of reaction modules comes from utilizing chemical reactions to perform operations \cite{vasic2020,buisman2009computing}. Consider the concentration of certain species at one time as the system input and concentration of certain species at another time as the output(input species and output species are usually different), chemical reaction networks can serve as a framework for calculation, and it has been proven that deterministic (mass-action) chemical kinetics is Turing universal \cite{fages2017strong}.
\par There are, as far as we know, two main ways of constructing chemical reactions to achieve a specific operation: One is choosing non-competitive(NC) CRNs whose equilibria are absolutely robust to reaction rates and kinetic rate law \cite{vasic2021programming}, so reaction networks can always achieve specific results regardless of the influence of parameters and initial values. The other one is regarding the operation to be implemented as the expression of ODEs at the equilibrium point \cite{vasic2020}, what operation the reaction network implements depends on kinetic assumption, parameters and initial values of ODEs and even speed of convergence.  Although the former has good robustness, it can realize a very narrow range of operations \cite{chalk2019composable}. While the latter can perform general operations, requiring elaborate design. This paper precisely addresses the problem of coupling reaction modules designed by the latter.
\par We first give an example about reaction modules:
\begin{example}~\\
reaction module 1:
    \begin{align*}
        X_{1} &\to X_{1} + X_{2}\ ,  \\
        X_{3} &\to X_{3} + X_{2}\ , \\
        X_{2} &\to \varnothing \ .  
    \end{align*}
reaction module 2:
    \begin{align*}
         X_{2} &\to X_{1} + X_{2}\ ,  \\
         X_{1} &\to \varnothing\ . 
    \end{align*}
\end{example}

When the reaction rate constant is exactly 1, we omit it by default, and $X_{i} \to \varnothing$ refers to outflow reaction. Considering that the implementation of complex operations requires multiple reaction modules to be coupled, we command that the concentrations of species as input remain constant under the module operation. So in reaction module 1, both $X_{1}$ and $X_{3}$ are input species and species $X_{2}$ is output; while in reaction module 2, the input is $X_{2}$ and output is $X_{1}$. 
Based on mass-action kinetics, we achieve the ODEs for the two reaction modules.\\
reaction module 1:
\begin{equation}
\begin{aligned}
\frac{\mathrm{d} x_{1}}{\mathrm{d} t} &= 0\ ,  \\
\frac{\mathrm{d} x_{2}}{\mathrm{d} t} &= x_{1} + x_{3} - x_{2}\ ,  \\
\frac{\mathrm{d} x_{3}}{\mathrm{d} t} &= 0\ . 
\end{aligned}
\end{equation}
reaction module 2:
\begin{equation}
\begin{aligned}
\frac{\mathrm{d} x_{1}}{\mathrm{d} t} &= x_{2} - x_{1}\ ,  \\
\frac{\mathrm{d} x_{2}}{\mathrm{d} t} &= 0\ .  
\end{aligned}
\end{equation}
These modules correspond exactly to the description of Add module and Id module by Vasic et al. \cite{vasic2020}. We let concentration of species $X_{3}$ play the role of constant value 1(make its initial value equal 1), then the alternation of these two modules realizes the operation instruction like $x_{1}+=1$. However, if we just put these two reaction modules together, the concentrations of $X_{1}$ and $X_{2}$ would continue to increase to infinity without interruption, which fails to reach our aim. This is the stage for chemical oscillators.

\subsection{Design requirements for chemical oscillators}
To be specific in our Example 2.2, we actually need to interrupt and stagger the progression of reactions in the two modules without changing the computational content of the respective module. So we construct two clock signals that are served by specific species $U$ and $V$, and add them to the two separate modules as catalyst. Then we get modified reaction modules:
\begin{example}
~\\
modified reaction module 1:
    \begin{align*}
        X_{1} + V &\to X_{1} + X_{2} + V\ , \\
        X_{3} + V&\to X_{3} + X_{2} + V\ , \\
        X_{2} + V&\to V\ . 
    \end{align*}
modified reaction module 2:
    \begin{align*}
         X_{2} + U&\to X_{1} + X_{2} + U\ , \\
         X_{1} + U&\to U\ . 
    \end{align*}
\end{example}    
Putting them together, the ODEs change:
\begin{equation}
\begin{aligned}
\frac{\mathrm{d} x_{1}}{\mathrm{d} t} &= (x_{2} - x_{1})u\ ,  \\
\frac{\mathrm{d} x_{2}}{\mathrm{d} t} &= (x_{1} + x_{3} -x_{2})v\ ,  \\
\frac{\mathrm{d} x_{3}}{\mathrm{d} t} &= 0\ . 
\end{aligned}
\end{equation}

\begin{figure}
  \centering
  \includegraphics[width=\columnwidth]{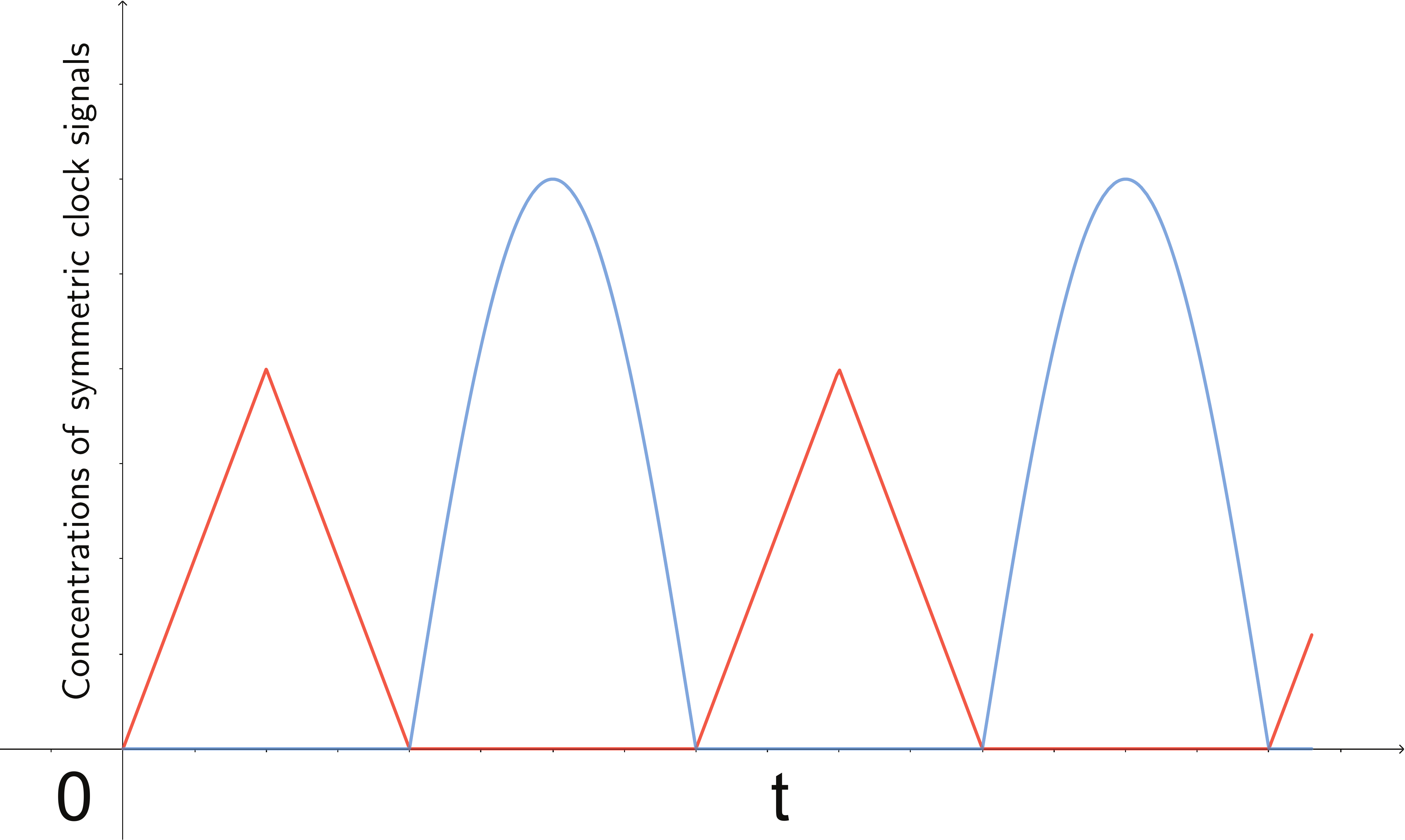}
  \hspace{2in} 
  \includegraphics[width=\columnwidth]{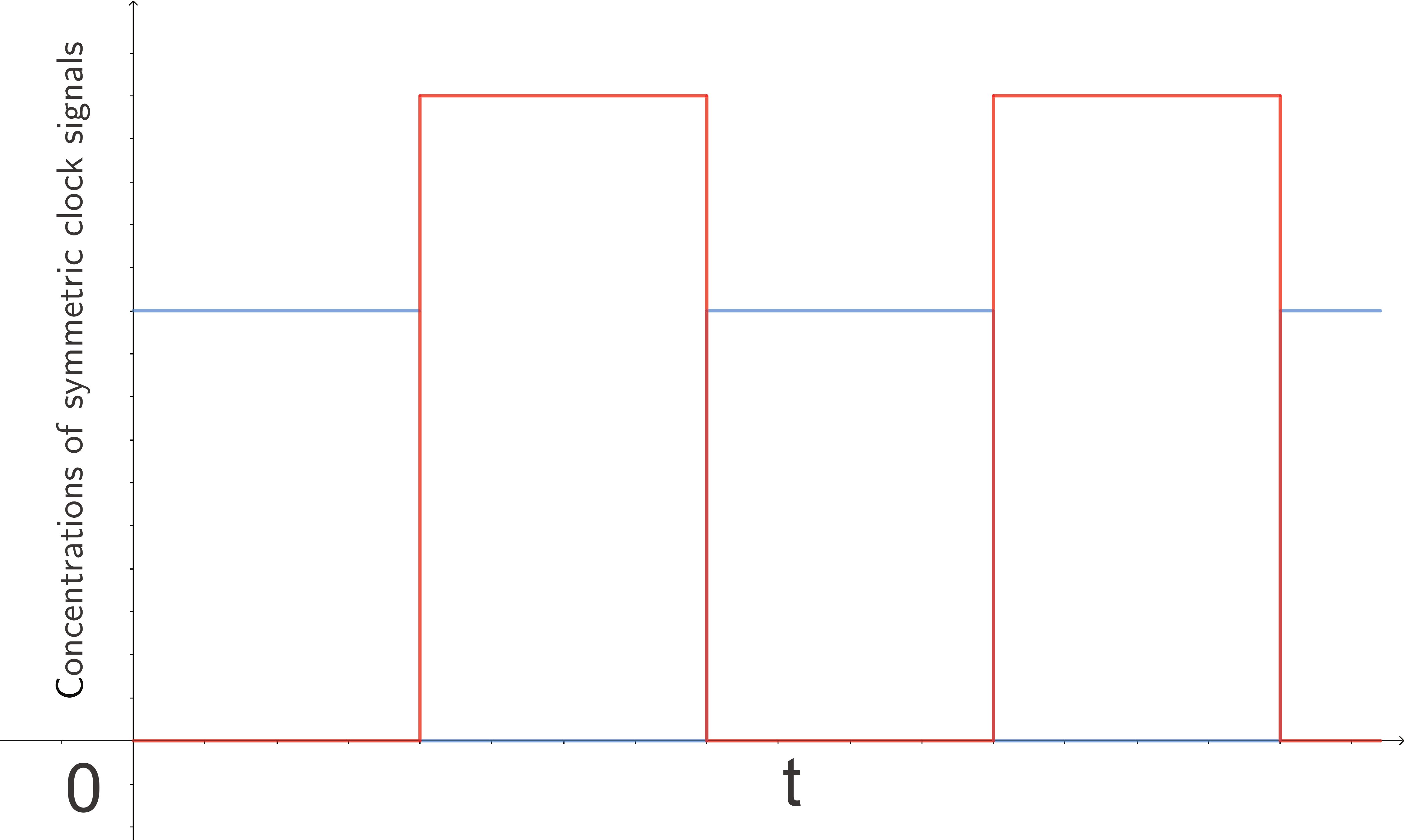}
  \caption{Two diagrams for symmetric clock signals.}
\end{figure}

We naturally use $u$ and $v$ to refer to the concentrations of $U$ and $V$. In order to shut down specific reaction module, concentration of the clock signal needs to stay near zero for some time, and then goes beyond zero to open the module again. Note that our ultimate goal is to have the two reaction modules alternating in cycles, and we treat this example as "Counter Model" which will be used for the spontaneous termination of loop.
Then the oscillatory structures of clock signal $U$ and $V$ should exhibit a certain degree of symmetry. We conclude these as following definitions.
\begin{definition}
	A species is called clock signal for that its concentration oscillates over time, which reaches zero or close enough during some part in a oscillation period and immediately goes beyond zero during the rest part. 
\end{definition}
\begin{definition}
	A pair of clock signals U and V are called symmetric for that when U oscillates to zero or close enough, V goes strictly beyond zero, and vice versa.
\end{definition}
We give two diagrams for symmetric clock signals in Fig.1.

Although the clock signals in both diagrams meet our expectations, We find that the former is more difficult to construct accurately than the latter. Concentrations of the two components will always cross at non-zero, leading to inevitable error. Our chemical oscillator can generate the latter form of clock signals, although they do not remain level at high amplitude, which does not affect the regulation of the reaction modules.

\section{Universal Chemical Oscillator based on Relaxation Oscillation}
This section we introduce the oscillation mechanism we choose and demonstrate its advantages for generating oscillator models. Note that advantages of our oscillator model are the transparency of mechanism and the independence of initial value selection, rather than independence of parameter selection.

\subsection{2-dimension Relaxation Oscillation}
\par Although there are various mechanisms for generating oscillations \cite{gonze2021goodwin,banaji2018inheritance,conradi2019,epstein1998introduction}, few of them can reach our requirements. We give up harmonic oscillators whose oscillatory structure is highly sensitive to the selection of initial values. Most of the limit cycle oscillations we know are difficult to produce oscillatory components that satisfy the definition of clock signals. So we choose relaxation oscillation as the basic mechanism for designing universal chemical oscillators.
\par Relaxation oscillation is a common type of oscillation in biochemical systems \cite{krupa2013network},which can be found from the Oregonator model \cite{field1974oscillations} and the Fitzhugh-Nagumo model \cite{fernandez2020symmetric} in different contexts and its dynamic behaviour has been studied in detail \cite{fernandez2020symmetric,krupa2001relaxation,grasman2012asymptotic}. We first give 2-dimensional structure of general relaxation oscillation model:
\begin{equation}\begin{aligned}
\epsilon \dot{x} &= f(x,y)\ ,  \\
\dot{y} &= g(x,y)\ ,  \\
x \in \mathbb{R},\ &y \in \mathbb{R},\ 0 < \epsilon \ll 1\ .
\end{aligned}\end{equation}
We give following hypothesis for the existence of relaxation oscillation adopted from \cite{krupa2001relaxation}.
\newtheorem{hypothesis}{Hypothesis}[section]
\begin{hypothesis}
	\begin{enumerate}
		\item The critical manifold is defined by $S\overset{\underset{\mathrm{def}}{}}{=} \left\{ \left ( x,y \right ) : f(x,y) = 0\right\}$ and it is S-shaped: Manifold S can be written in the form $y = \varphi (x)$ and the smooth function $\varphi$ has precisely two extreme points, one non-degenerate minimum $x = x_{m}$ and one non-degenerate maximum$x= x_{M}$.
		The two points divide critical manifold $S$ into three parts: $S_{l}$, $S_{m}$ and $S_{r}$:
		\begin{align*}
		S_{l} &= \left\{ (x,\varphi (x)): x < x_{m}\right\}\ ,  \\
		S_{m} &= \left\{ (x,\varphi (x)): x_{m} < x < x_{M}\right\}\ ,  \\
		S_{r} &= \left\{ (x,\varphi (x)): x > x_{M}\right\}\ . 
		\end{align*}
		\item $S_{l}$ and $S_{r}$ are attracting, i.e. $\frac{\partial f}{\partial x} < 0$ on $S_{l}$ and $S_{r}$, while $S_{m}$ is repelling, i.e. $\frac{\partial f}{\partial x} > 0$ on $S_{m}$.
		\item Both extreme points satisfy conditions:
		\par $\frac{\partial^2 f}{\partial x^2}(x_{0},y_{0}) \neq 0$,  $\frac{\partial f}{\partial y}(x_{0},y_{0}) \neq 0$, $ g(x_{0},y_{0}) \neq 0$.
		\item The slow flow on $S_{l}$ satisfies $\dot{x} > 0$ and the slow flow on $S_{r}$ satisfies $\dot{x} < 0$.
	\end{enumerate}
\end{hypothesis}
Hypothesis 3.1 actually describes the phase plane portrait of system (4), which can be viewed as Fig.2. The coordinates of points are $A(x_{A},y_{M})$, $B(x_{m},y_{m})$, $C(x_{C},y_{m})$, $D(x_{M},y_{M})$. Then we define a singular trajectory $\Gamma$ as $\Gamma = \left\{ \left ( x,y \right ) \in S_{l}:x_{A} < x < x_{m} \right\} \cup \left\{ \left ( x,y_{m} \right ):x_{m} < x < x_{C}\right\} \cup \left\{ \left ( x,y \right ) \in S_{r}: x_{M} < x < x_{C}\right\} \cup \left\{ \left ( x,y_{M} \right ): x_{A} < x < x_{M}\right\}$. 
\begin{figure}[!t]
	\centerline{\includegraphics[width=\columnwidth]{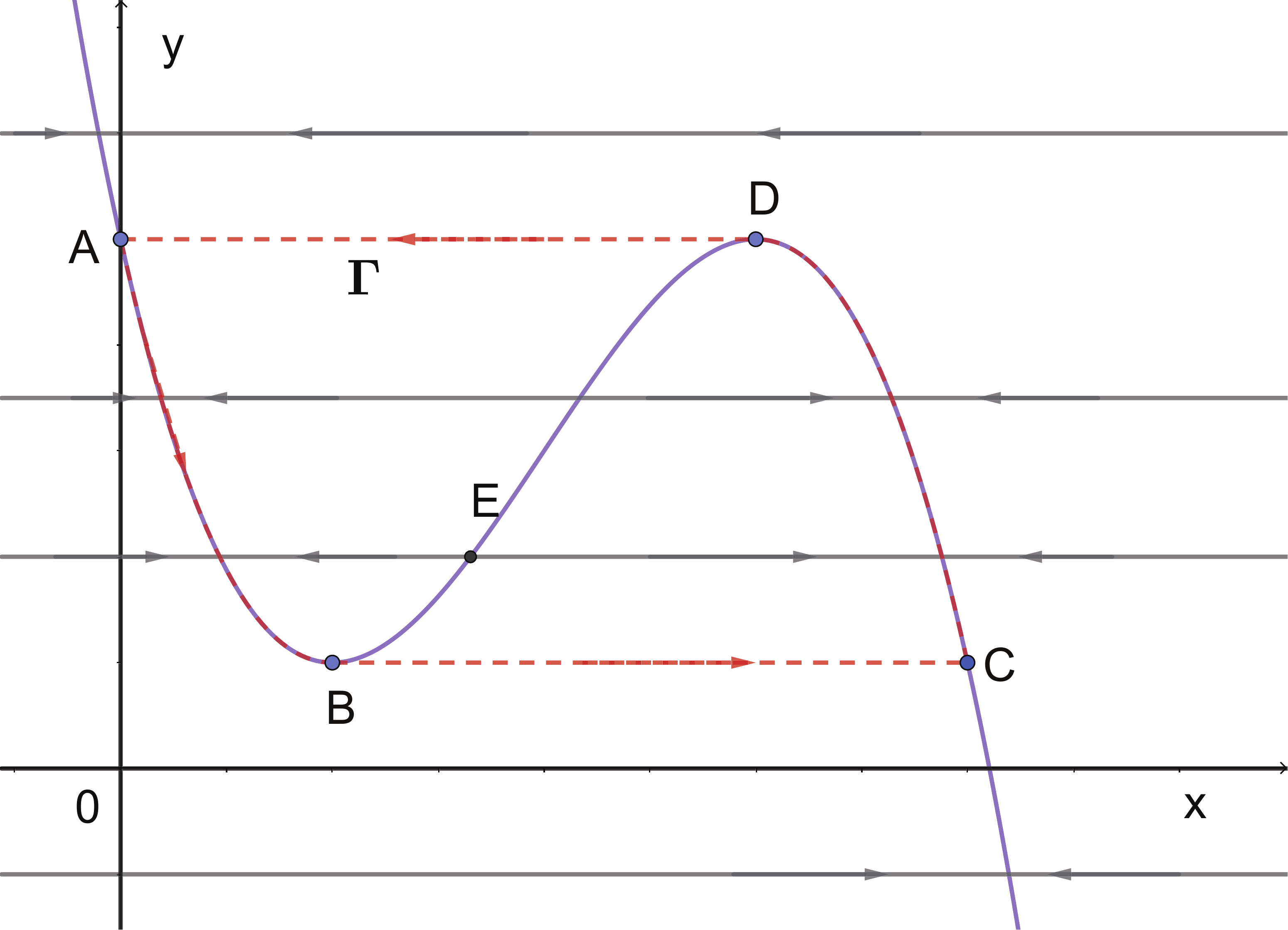}}
	\caption{Phase plane portrait of system (4) based on Hypothesis 3.1.}
	\label{fig2}
\end{figure}
\newtheorem{lemma}{Lemma}[section]
\begin{lemma}
	Assume Hypothesis 3.1. Then for sufficiently small $\epsilon$, there exists a unique limit cycle $\Gamma_{\epsilon}$ lying in a small tubular neighborhood of $\Gamma$. The cycle $\Gamma_{\epsilon}$ is strongly attracting and as $\epsilon \to 0$, the cycle $\Gamma_{\epsilon}$ approaches $\Gamma$ in the Hausdorff distance.
\end{lemma} 
This lemma is the famous result as THEOREM 2.1 in \cite{krupa2001relaxation}, applying Fenichel Slow manifold theory and fundamental knowledge of geometric singular perturbation can easily prove it, we do not repeat this here. Besides, M. Krupa et al. mentioned more complex issues such as canard explosion, brief appearance and disappearance of limit cycles caused by Hopf bifurcation \cite{krupa2001relaxation}, which are not the focus of this paper. Actually we are just interested in existence and robustness of the limit cycle $\Gamma_{\epsilon}$ named relaxation oscillation.
\par Note that we value the oscillatory components in terms of concentrations of species, so the limit cycle $\Gamma_{\epsilon}$ should be limited in the first quadrant of phase plane portrait. And for the convenience of designing oscillator, we also limit function $g$ as linear function $g\left ( x,y \right ) = x + \mu y + \lambda $, $\mu \in \mathbb{R}$, $\lambda \in \mathbb{R}$.

\begin{lemma}
	Assume the Hypothesis 3.1 and $g\left ( x,y \right ) = x + \mu y + \lambda $, $\mu \in \mathbb{R}$, $\lambda \in \mathbb{R}$, then system(4) has and only has equilibrium points on the manifold $S_{m}$. Moreover, we suppose that $\mu \leq 0$ and $\left |  \mu\right |$ is small enough, then system(4) can only have one unique equilibrium point E, and E is unstable.
\end{lemma}	

\newtheorem*{proof}{Proof}
\begin{proof}
The last one in Hypothesis 3.1 says that $\dot{x} > 0$ on $S_{l}$ and $\dot{x} < 0$ on $S_{r}$ while $\dot{x} = \frac{g\left ( x,\varphi \left ( x \right ) \right )}{\varphi^{'} \left ( x \right)}$. The sign of $\varphi^{'} \left ( x \right)$ on $S_{l}$ is same as the one on $S_{r}$, so the graph of $g\left ( x,y \right ) = 0$ must be between $S_{l}$ and $S_{r}$, which leads to equilibrium points on $S_{m}$. While equilibrium points on $S_{l}$ or $S_{r}$ would destroy the consistent result of the last one in Hypothesis 3.1, system(4) has and only has equilibrium points on $S_{m}$. \\
Moreover, suppose that $\mu \leq 0$ and $\left |  \mu\right |$ is small enough, then system(4) has a unique equilibrium point E, and E must lie on manifold $S_{m}$. We define the part of $S_{m}$ below E as $S_{mb}$, and the part above E as $S_{ma}$. Then the sign of $g\left ( x,\varphi \left ( x \right ) \right )$ on $S_{mb}$ is same as the one on $S_{l}$, while signs of $\varphi^{'} \left ( x \right)$ are different. So the sign of $\dot{x}$ on $S_{mb}$ is negative. Similarly, $\dot{x} > 0$ on $S_{ma}$. This means that initial points close to E on $S_{m}$ would stay away from E. Combined with the second one in Hypothesis 3.1, equilibrium point E is unstable.  $\hfill\blacksquare$
\end{proof}
  
 Based on Hypothesis 3.1 and $g\left ( x,y \right ) = x + \mu y + \lambda $, system(4) can actually have odd equilibrium points on $S_{m}$, which would cause strange dynamics and complicate model analysis. However, our focus is not to analyze the dynamical properties of system(4) in any case. We limit $g$ as linear function along with range of parameter $\mu$
 in order to simplify complexity of our oscillator model and achieve the desired dynamic behaviour.

We conclude Theorem 3.1 as follows:
\newtheorem{theorem}{Theorem}[section]
\begin{theorem}
    Assume Hypothesis 3.1 and add that:
	\begin{enumerate}
		\item Singular trajectory $\Gamma$ with its small tubular neighborhood $U$ strictly lies in the first quadrant.
		\item $g\left ( x,y \right ) = x + \mu y + \lambda $, $\mu \leq 0$ and $\left |  \mu\right |$ is small enough.
\end{enumerate}
    \par Then for sufficiently small $\epsilon$, relaxation oscillation exists in the first quadrant of phase plane portrait and furthermore, all of trajectories starting from this quadrant except the equilibrium point reach the limit cycle $\Gamma_{\epsilon}$ finally.
\end{theorem}

\begin{proof}
 Lemma 3.1 ensures the existence of $\Gamma_{\epsilon}$, which is closely related to the slow manifold. In Fenichel Slow Manifold Theorem \cite{fenichel1979geometric}, slow manifold $M_{\epsilon}$ falls in the $O\left(\epsilon\right)$ neighborhood of normal hyperbolic manifold $M$. So cycle $\Gamma_{\epsilon}$ can be viewed as perturbation of trajectory $\Gamma$ under parameter $\epsilon$. For sufficiently small $\epsilon$, the relaxation oscillation $\Gamma_{\epsilon}$ approaches $\Gamma$ in the Hausdorff distance and exists in the first quadrant.  
 Lemma 3.2 shows that the invariant set of system(4) consist of singular trajectory $\Gamma_{\epsilon}$ and equilibrium point E. While E is unstable, the unique stable invariant set in first quadrant is $\Gamma_{\epsilon}$. 
As is shown in Fig.2, trajectory with initial points on $S_{l} \cup S_{r}$ or nearby goes along the cycle $\Gamma_{\epsilon}$ immediately, while trajectory with initial points somewhere else in the first quadrant except E pours along horizontal flows at the beginning until reaching neighborhood of $S_{l}$ or $S_{r}$, then oscillating along $\Gamma_{\epsilon}$. So all of trajectories starting from first quadrant reach the limit cycle $\Gamma_{\epsilon}$ finally.   $\hfill\blacksquare$
\end{proof}

\par The sufficiently small parameter $\epsilon$ leads to two time scales in system(4), and Theorem 3.1 is actually the classical conclusion of fast-slow system. The main contribution of $\epsilon$ is making trajectory with initial points away from the neighborhood of $S_{l}$ or $S_{r}$ converge to left part or right part of $\Gamma_{\epsilon}$ quickly, which results in abrupt transitions between the phases of $x$.
\par Note that we show the independence between relaxation oscillation structure and initial value selection. However, we would not talk about robustness associated with parameters except $\epsilon$ which may exist in system(4), for that properties of oscillation such as amplitude and period are strictly dependent on these parameters.
\par We give an example with the form of relaxation oscillation based on system(4).
\begin{example}
 \begin{equation}\begin{aligned}
 \frac{\mathrm{d} x}{\mathrm{d} t} &= \left (-x^3+6x^2-9x+5-y \right ) / \epsilon\ ,\\ 
 \frac{\mathrm{d} y}{\mathrm{d} t} &= x - 2\ .
 \end{aligned}\end{equation}
\end{example}
It is easy to verify that choice of function $f$ and $g$ satisfies the assumptions of Theorem 3.1, with $\epsilon=0.001$ and initial point$\left (1,1 \right )$, we get the simulation result in Fig.3.
 \begin{figure}[!t]
 	\centerline{\includegraphics[width=\columnwidth]{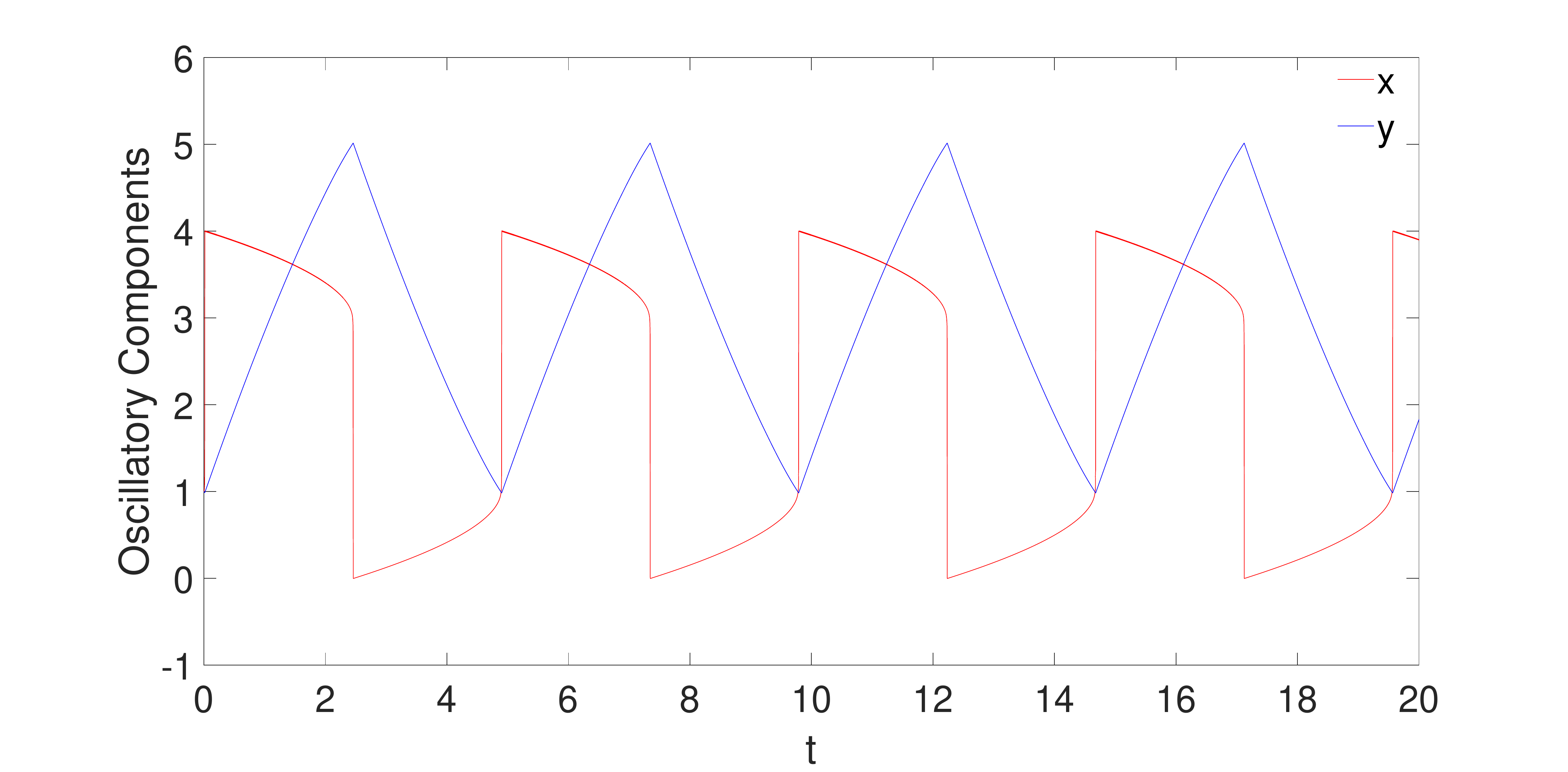}}
 	\caption{Simulation result of system(5)}.
 	\label{fig3}
 \end{figure}
In the previous section we have shown how to use mass-action kinetics to model chemical reaction networks as ODEs. However, not all forms of ODEs can be converted back into chemical reactions \cite{hangos2011mass}. In simple terms, if there is a negative term in the ODE expression corresponding to $x$, the value of $x$ will decrease, then the species $X$ corresponding to the chemical reactions should exist as the reactant(note that $\dot{x}$ refers to concentration of species $X$). Therefore, the negative term in the ODE expression corresponding to $x$ must factor in $x$. Example 3.1 actually makes no sense in CRN.
Given this, we modify the functions in Example 3.1 as following:
\begin{example}
 \begin{equation}\begin{aligned}
 \frac{\mathrm{d} x}{\mathrm{d} t} &= \left (-x^3+6x^2-9x+5-y \right ) x/ \epsilon\ ,\\ 
 \frac{\mathrm{d} y}{\mathrm{d} t} &= \left (x - 2 \right)y\ . 
 \end{aligned}\end{equation}
\end{example}
Although the modification does not destroy the structure of critical manifold, two additional equilibrium points $(0,0)$ and $(x_{0},0)$($-x_{0}^3+6x_{0}^2-9x_{0}+5=0$) emerge, which are saddle points. We just have to avoid the points on the axes as initial values, then Theorem 3.1 still holds. With the same values of parameter and initial point as Example 3.1, we give the simulation for Example 3.2 in Fig.4.
\begin{figure}[!t]
 	\centerline{\includegraphics[width=\columnwidth]{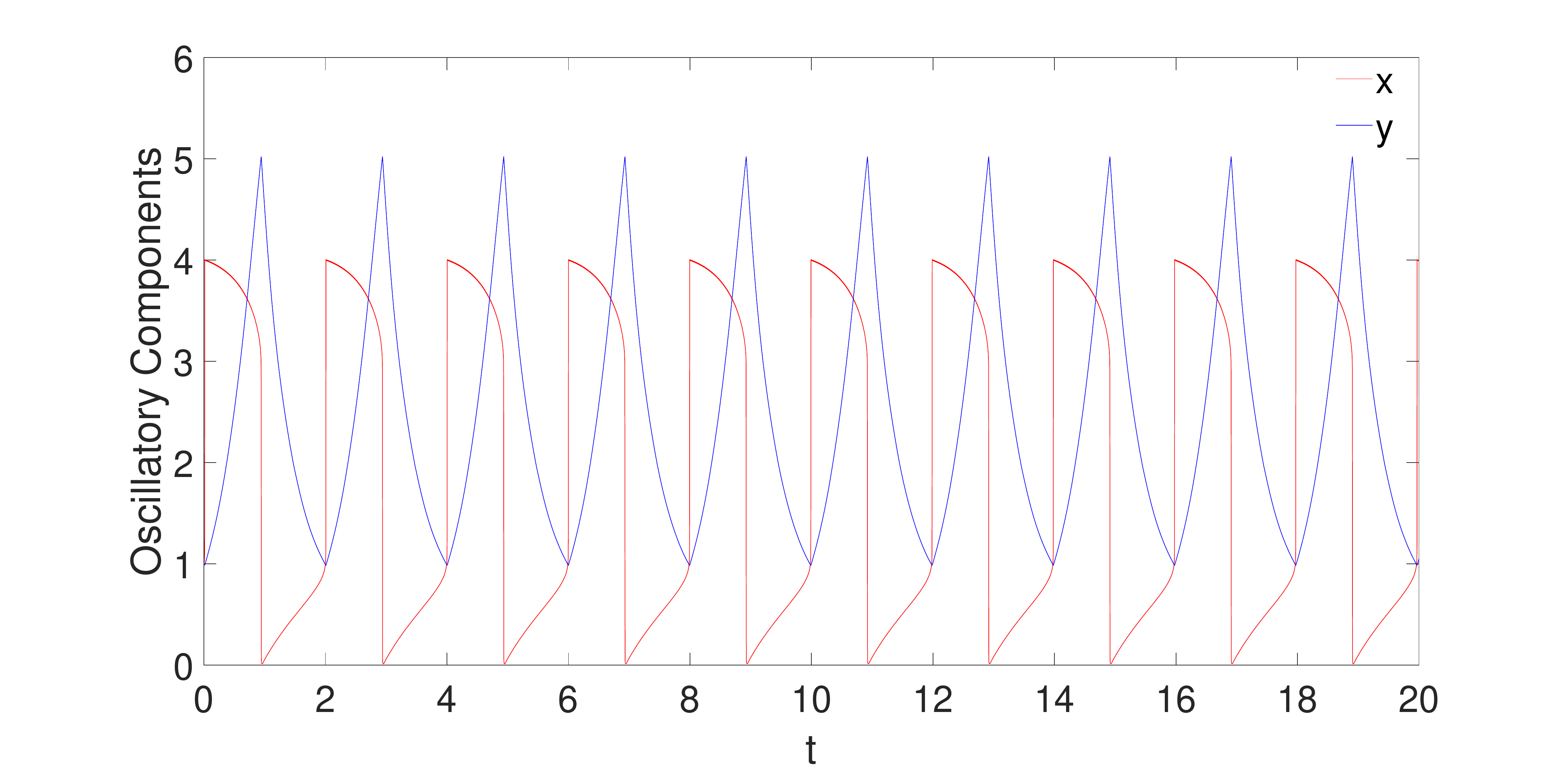}}
 	\caption{Simulation result of system(6)}.
 	\label{fig4}
 \end{figure}
 
We supplement the corresponding CRNs as follows:
\begin{align*}
    4X&\overset{1/\epsilon }{\rightarrow}3X\ ,\\
3X&\overset{6/\epsilon }{\rightarrow}4X\ , \\
2X&\overset{9/\epsilon }{\rightarrow}X\ , \\
X&\overset{5/\epsilon }{\rightarrow}2X\ , \\
X+Y&\overset{1/\epsilon }{\rightarrow}Y\ , \\
X+Y&\overset{1}{\rightarrow}2Y\ , \\
2Y&\overset{2}{\rightarrow}Y\ .
\end{align*}

\par Note that our goal is to find a simplest model which can generate a pair of symmetric clock signals as we define in previous section based on relaxation oscillation, but the oscillatory components in system(4) can not reach our requirement for that neither $x$ nor $y$ could stay near zero enough for some time and they are not actually symmetric. So we need to use the oscillatory structure of $x$ to construct new pair of components to act as symmetric clock signals by coupling $x$ unidirectional to the module we will introduce next.

\subsection{Coupled with Modified Truncated Subtraction Module} 
We first introduce the truncated subtraction module mentioned in \cite{buisman2009computing,vasic2020}:
\begin{align*}
X_{1} &\to X_{1}+X_{3}\ ,\\
X_{2} &\to X_{2}+X_{4}\ , \\
X_{3} &\to \varnothing\ ,  \\
X_{3}+X_{4} &\to \varnothing\ . 
\end{align*}
which computes truncated subtraction corresponding to equilibrium of ODEs:
\begin{equation*}
x_{3} = \begin{cases}
x_{1}-x_{2}, & \text{ if $x_{1}>x_{2}$}  \\
0, & \text{ otherwise } 
\end{cases}
\end{equation*}
Based on this, we add the outflow reaction of species $X_{4}$ and consider the influence of reaction rate of the last reaction in order to treat output species $X_{3}$ and $X_{4}$ as our symmetric clock signals:
\begin{equation}\begin{aligned}
X_{1} &\to X_{1}+X_{3}\ ,\\
X_{2} &\to X_{2}+X_{4}\ , \\
X_{3} &\to \varnothing\ ,  \\
X_{4} &\to \varnothing\ ,  \\
X_{3}+X_{4} &\overset{c}{\rightarrow} \varnothing\ . 
\end{aligned}\end{equation}
ODEs of reaction network(7) express as follows:
 \begin{equation}\begin{aligned}
 \frac{\mathrm{d} x_{3}}{\mathrm{d} t} &= x_{1}-x_{3}-cx_{3}x_{4}\ , \\ 
 \frac{\mathrm{d} x_{4}}{\mathrm{d} t} &= x_{2}-x_{4}-cx_{3}x_{4}\ . 
 \end{aligned}\end{equation}
Take $x_{1}$ and $x_{2}$ as inputs, if we value parameter $c$ as zero i.e. species $X_{3}$ does not couple with $X_{4}$, then reaction network (7) just load value of $x_{1}$ and $x_{2}$ separately into $x_{3}$ and $x_{4}$. The coupling parameter $c$ complicates the dynamic behaviour that the ODEs (8) can induce, which we would analysis in detail in next section.
\par Till now, we conclude our universal oscillator model in the context of ODEs:
\begin{equation}
    \begin{aligned}
    \frac{\mathrm{d} x}{\mathrm{d} t} &= \eta_{1}\eta_{2}(f(x)-y)x/\epsilon\ , \\
    \frac{\mathrm{d} y}{\mathrm{d} t} &= \eta_{1}\eta_{2}(x+\mu y+\lambda)y\ , \\
    \frac{\mathrm{d} u}{\mathrm{d} t} &= \eta_{2}(p-u-cuv)\ , \\
    \frac{\mathrm{d} v}{\mathrm{d} t} &= \eta_{2}(x-v-cuv)\ .
    \end{aligned}
\end{equation}
Where ODEs of $x$ and $y$ are just combination of system (4) and assumptions in Theorem 3.1 and the requirement for transformation from ODEs to CRNs, and we couple the expression of relaxation oscillation on $x$ with modified truncated subtraction module (8): Substitute input species $x$ to $x_{2}$ and utilize a constant $p$ as $x_{1}$. Parameters $\eta_{1}$ and $\eta_{2}$ are used to regulate the period of variables and can be inserted into corresponding reaction rates in CRNs.
\par We treat the value of $u$ and $v$ as the output and come back to Example 3.2 to show that corresponding species $U$ and $V$ could act as symmetric clock signals that we want.
\begin{example}
 \begin{equation}
    \begin{aligned}
    \frac{\mathrm{d} x}{\mathrm{d} t} &= \eta_{1}\eta_{2}(-x^3+6x^2-9x+5-y)x/\epsilon\ , \\
    \frac{\mathrm{d} y}{\mathrm{d} t} &= \eta_{1}\eta_{2}(x-\rho)y\ , \\
    \frac{\mathrm{d} u}{\mathrm{d} t} &= \eta_{2}(p-u-cuv)\ , \\
    \frac{\mathrm{d} v}{\mathrm{d} t} &= \eta_{2}(x-v-cuv)\ .
    \end{aligned}
\end{equation}
With $\eta_{1}=0.01$, $\eta_{2}=10$, $\epsilon=0.001$, $\rho=2$, $p=2$, $c=400$ and initial point $(1,1,0,0)$, we get simulation result as Fig.5.
\end{example}
 \begin{figure}[!t]
 	\centerline{\includegraphics[width=\columnwidth]{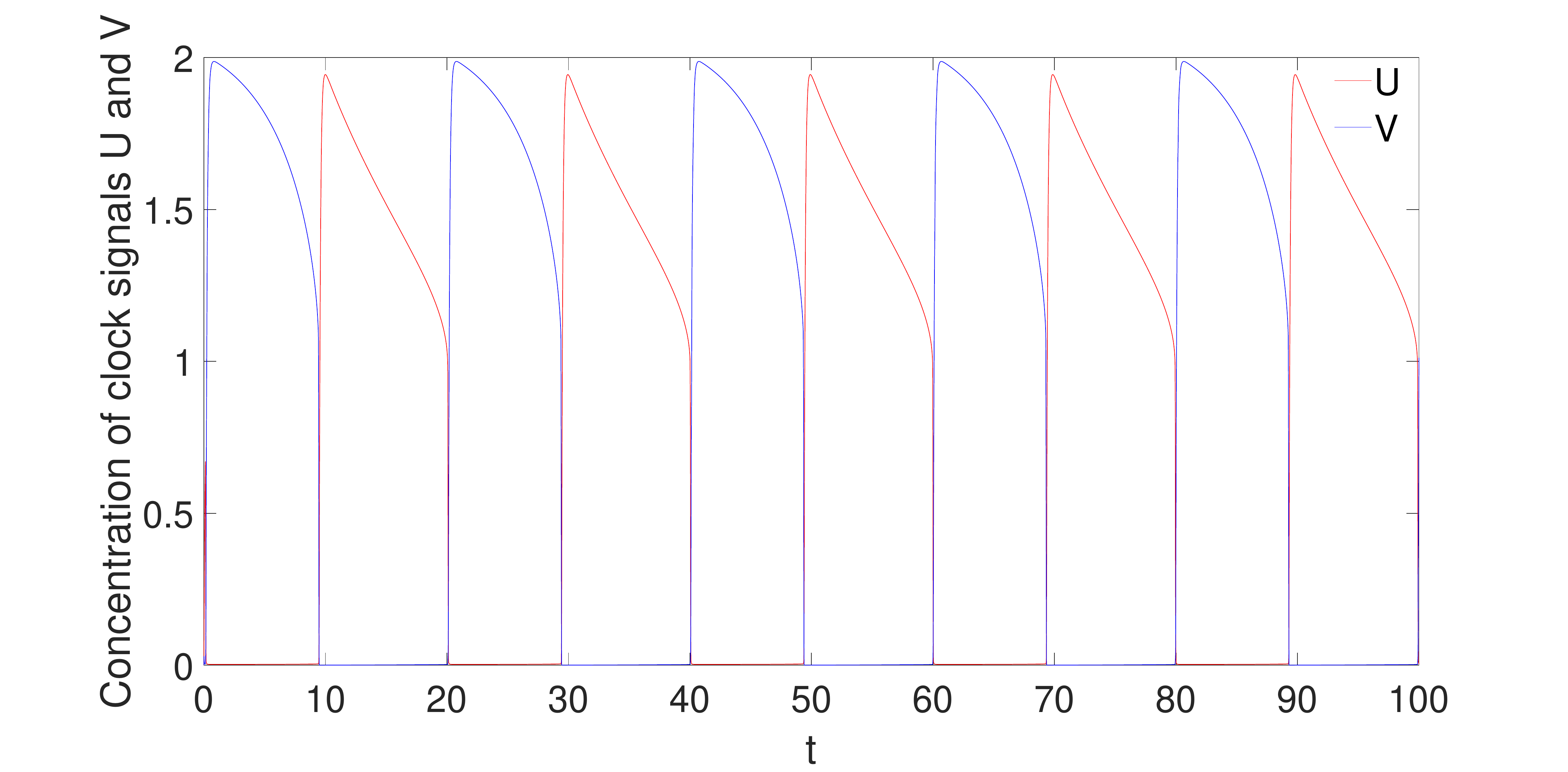}}
 	\caption{Simulation result of system(10)}.
 	\label{fig5}
 \end{figure}
\par Furthermore, return to the modified reaction module 1 and modified reaction module 2 in Example 2.3 and combine ODEs in system (3) and Example 3.3, we rewrite the whole ODEs as follows:
 \begin{equation}
    \begin{aligned}
    \frac{\mathrm{d} x}{\mathrm{d} t} &= \eta_{1}\eta_{2}(-x^3+6x^2-9x+5-y)x/\epsilon\ , \\
    \frac{\mathrm{d} y}{\mathrm{d} t} &= \eta_{1}\eta_{2}(x-\rho)y\ , \\
    \frac{\mathrm{d} u}{\mathrm{d} t} &= \eta_{2}(p-u-cuv)\ , \\
    \frac{\mathrm{d} v}{\mathrm{d} t} &= \eta_{2}(x-v-cuv)\ , \\
    \frac{\mathrm{d} x_{1}}{\mathrm{d} t} &= \eta_{3}(x_{2}-x_{1})u\ , \\
    \frac{\mathrm{d} x_{2}}{\mathrm{d} t} &= \eta_{3}(x_{1}+1-x_{2})v\ .
    \end{aligned}
\end{equation}
We add another parameter $\eta_{3}$ to the last two equations in order to ensure the accuracy of adjustment by species $U$ and $V$, and substitute the concentration of catalyst $X_{3}$ just as constant one. Choose $\eta_{3}=0.35$ and initial values of both $x_{1}$ and $x_{2}$ as zero, we get simulation of the Counter Model in Fig.6. Note that although this model was originally designed to perform the operation instruction like $x_{1}+=1$, values of $x_{1}$ and $x_{2}$ can both play the role of counter, with only difference in phase.
 \begin{figure}[!t]
 	\centerline{\includegraphics[width=\columnwidth]{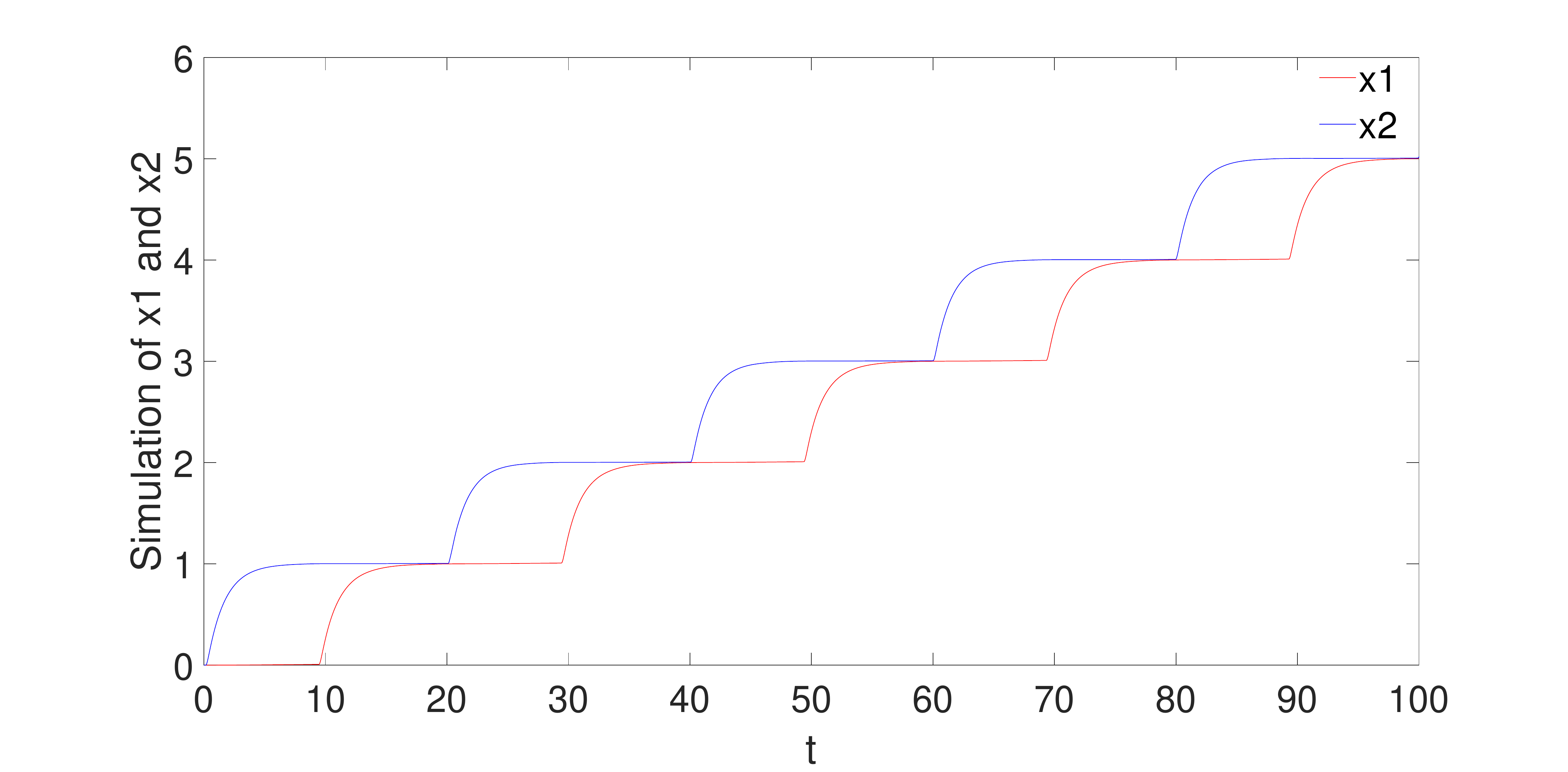}}
 	\caption{Simulation result of the Counter Model with components $x_{1}$ and $x_{2}$}.
 	\label{fig6}
 \end{figure}

\section{Analysis of Dynamic Behaviours of Universal Oscillator Model}
In this section we will explain the choice of parameters of the universal oscillator model in detail and show different dynamic behaviours the model can exhibit.
\subsection{Parameters in isolated system $x-y$ and $u-v$}
\par Let's firstly focus on system (9). The function $f(x)$ actually refers to critical manifold $S$ mentioned in section \uppercase\expandafter{\romannumeral3}, small enough $\epsilon$ results in different timescales between $x$ and $y$, which is foundation of analysis on relaxation oscillation. We choose $f(x)$ as a S-shaped function\cite{krupa2001relaxation} that lies in the first quadrant and the critical manifold $S$ has a unique intersection with the straight line $x+\mu y+\lambda$ on its middle segment i.e. the manifold $S_{m}$. We find that the structure in this form is general \cite{krupa2001relaxation,fernandez2020symmetric} and has its biochemical correspondence such as the FitzHugh-Nagumo system and the Oregonator model \cite{tyson1980target}. While in Example 3.3, we use cubic function to act as function $f(x)$ not only because polynomials are directly related to mass-action kinetics, but also in order to demonstrate the generality of our oscillator model. Note that the parameter $\rho$ should be between the two fold points of $f(x)$ i.e. $1<\rho<3$ for existence of relaxation oscillation, and if the equilibrium point of system $x-y$ lies close enough to the fold points($(1,1)$ and $(3,5)$ in Example 3.3), then there will be complex oscillations such as canard exposition and mixed-mode oscillation \cite{wechselberger2005existence}, which are not what we want. So we choose $1<\rho<3$ and let $\rho$ keep some distance with the two endpoint values. \\
\par Then in system $u-v$, isolate this system from system $x-y$ and assume that $x$ in the expression of $\frac{\mathrm{d} v}{\mathrm{d} t}$ is constant. When the value of parameter $c$ is large enough, $u$ and $v$ actually output the truncated subtraction between the value $p$ and $x$.
\begin{lemma}
    Assume parameter $c$ is large enough, then the system $u-v$ converges to the following approximate equilibrium depend on the magnitude of $x$ and $p$:
    \begin{equation*}
u = \begin{cases}
p-x, & \text{ if $p>x$}  \\
0, & \text{ otherwise } 
\end{cases}
\end{equation*}
\begin{equation*}
v = \begin{cases}
0, & \text{ if $p>x$}  \\
x-p, & \text{ otherwise } 
\end{cases}
\end{equation*}
\end{lemma}
\begin{proof}
 The equilibrium of this system actually expresses as $p-u-cuv=0$ and $x-v-cuv=0$. After a simple substitution, we get $v^2+(p-x+\frac{1}{c})v-\frac{x}{c}=0$. Since the value of $c$ is large enough, we can simplify it into $v^2+(p-x)v=0$, which has two solutions as $v=0$ and $v=x-p$, corresponding to the magnitude of $x$ and $p$. Situation of $u$ is similar.  $\hfill\blacksquare$
\end{proof}
\par Give $x$ back to oscillator as relaxation oscillation, how $u$ and $v$ follow the periodic oscillation of x to produce similar periodic behaviour depends not only on equilibrium of the isolated system $u-v$ with constant $x$, but also convergence speed of $u$ and $v$. Actually, equilibrium can just tell us the long term behaviour of an ODE system, while the periodic change in $u$ and $v$ is a real-time response to oscillation of input $x$. So restrict to the isolated system $u-v$, we first give a lemma on its exponential convergence:
\begin{lemma}
    For the isolated system $u-v$ as follows:
    \begin{equation}\begin{aligned}
    \frac{\mathrm{d} u}{\mathrm{d} t} &= \eta_{2}(p-u-cuv)\ , \\
    \frac{\mathrm{d} v}{\mathrm{d} t} &= \eta_{2}(x-v-cuv)\ . \\
    \end{aligned}\end{equation}
    Parameter $p$ and $x$ are positive constant and different from each other, $\eta_{2}>0$, $c>0$ and $c$ is large enough. Then the system converges to approximate equilibrium $(p-x,0)$ or $(0,x-p)$ at exponential speed.
\end{lemma}
\begin{proof}
The approximate equilibrium is shown in Lemma 4.1, we just focus on exponential convergence of this system. It is obvious that $u-v$ converges to $p-x$ at exponential speed for that $\frac{\mathrm{d} (u-v)}{\mathrm{d} t} = \eta_{2}((p-x)-(u-v))$. Then we can find $k>0$ and $\lambda>0$ satisfying $\left |  (u-v)-(p-x)\right |< ke^{-\lambda t}$ i.e. $(p-x)-u-ke^{-\lambda t}<-v<(p-x)-u+ke^{-\lambda t}$. Substitute into expression of $\frac{\mathrm{d} u}{\mathrm{d} t}$, we get 
\begin{equation}\begin{aligned}
\frac{\mathrm{d} u}{\mathrm{d} t}&<\eta_{2}(p-u-cu^{2}+c(p-x)u+cke^{-\lambda t}u)\\
&<\eta_{2}(-cu^2+(c(p-x)+ck-1)u+p)\ .
\end{aligned}\end{equation}
Without loss of generality, we only consider the case $p>x$($p<x$ is similar).
Then for following system:
\begin{equation*}
    \frac{\mathrm{d} u}{\mathrm{d} t}= \eta_{2}(-cu^2+(c(p-x)+ck-1)u+p)\ ,
\end{equation*}
where $-c<0$, $(c(p-x)+ck-1>0$, $p>0$, function in the right of the ODE must have two zero roots which we name as $u_{l}$ and $u_{r}$, and $u_{l}<0<u_{r}$(A large enough $c$ makes the left root $u_{l}$ close to zero). Then ODE above can be transformed into 
\begin{equation}
    \frac{\mathrm{d} u}{\mathrm{d} t}= -c\eta_{2}(u-u_{l})(u-u_{r})\ ,
\end{equation}
which has solution as $u=u_{r}+\frac{u_{r}-u_{l}}{me^{-c(u_{l}-u_{r})\eta_{2}t}-1}$. So ODE(14) converges to $u=u_{r}$ at exponential speed. Then ODE $\frac{\mathrm{d} u}{\mathrm{d} t} = \eta_{2}(p-u-cuv)$ also has exponential convergence by inequality (13). Same for the convergence of $v$.         $\hfill\blacksquare$
\end{proof}

\par Expression at equilibrium of $u$ and $v$ is actually our requirement for clock signal in one period. To alternate the values of $u$ and $v$, we must choose $p$ between high and low amplitudes of oscillator $x$. In Example 3.3, high amplitude of $x$ is between 4 and 3, while the low amplitude is between 0 and 1, so the range of $p$ is between 1 and 3. Fix the other parameters and choose $p=0.5$, $p=2.5$, $p=3.5$, it is obvious that when $p$ falls out of the range, oscillations of $u$ and $v$ occur in intersecting segments that are not both zero, which destroy the symmetry of $u$ and $v$ as Fig.7.
\begin{figure}
  \centering
  \includegraphics[width=\columnwidth]{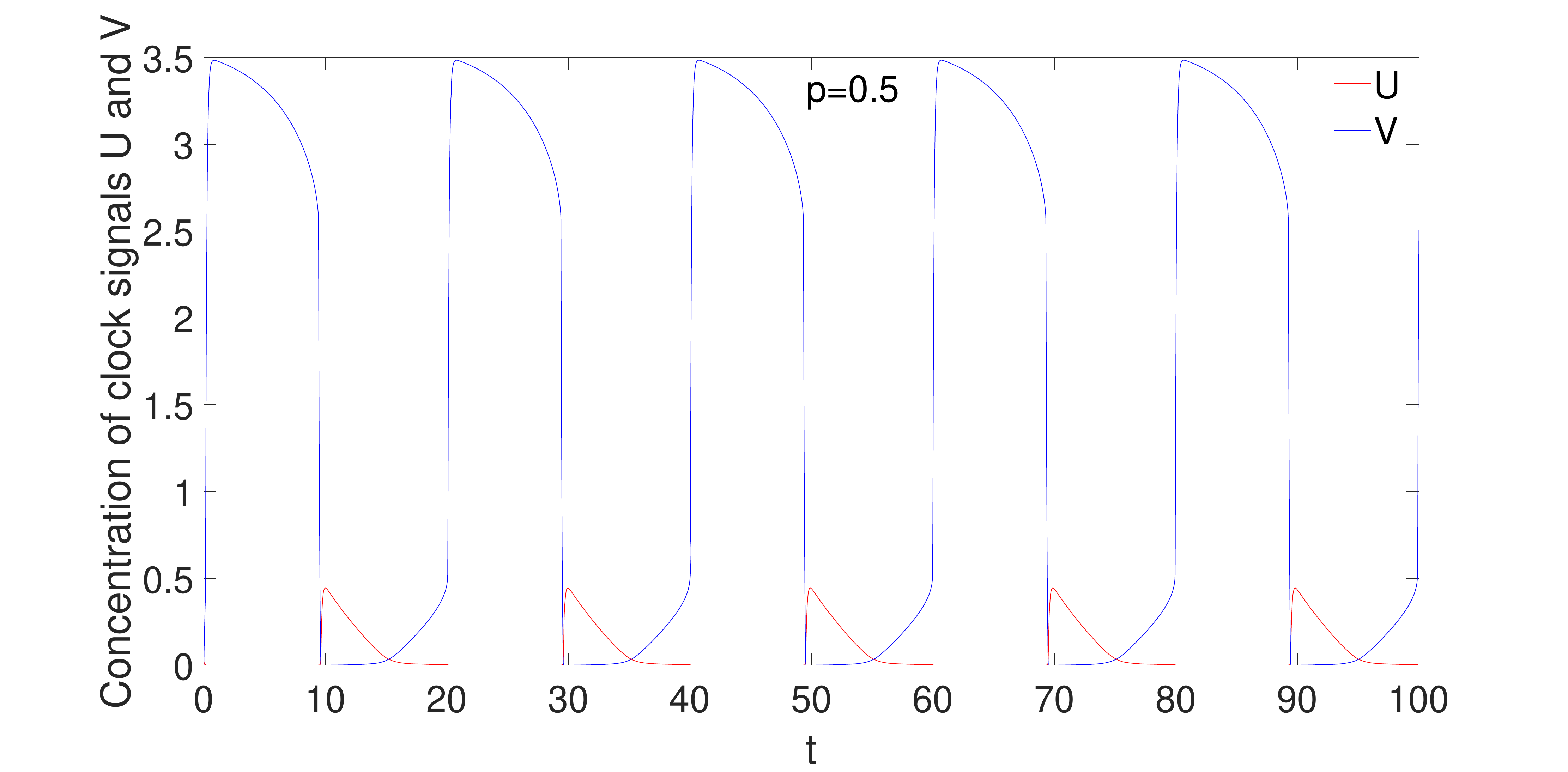}
  \hspace{2in} 
  \includegraphics[width=\columnwidth]{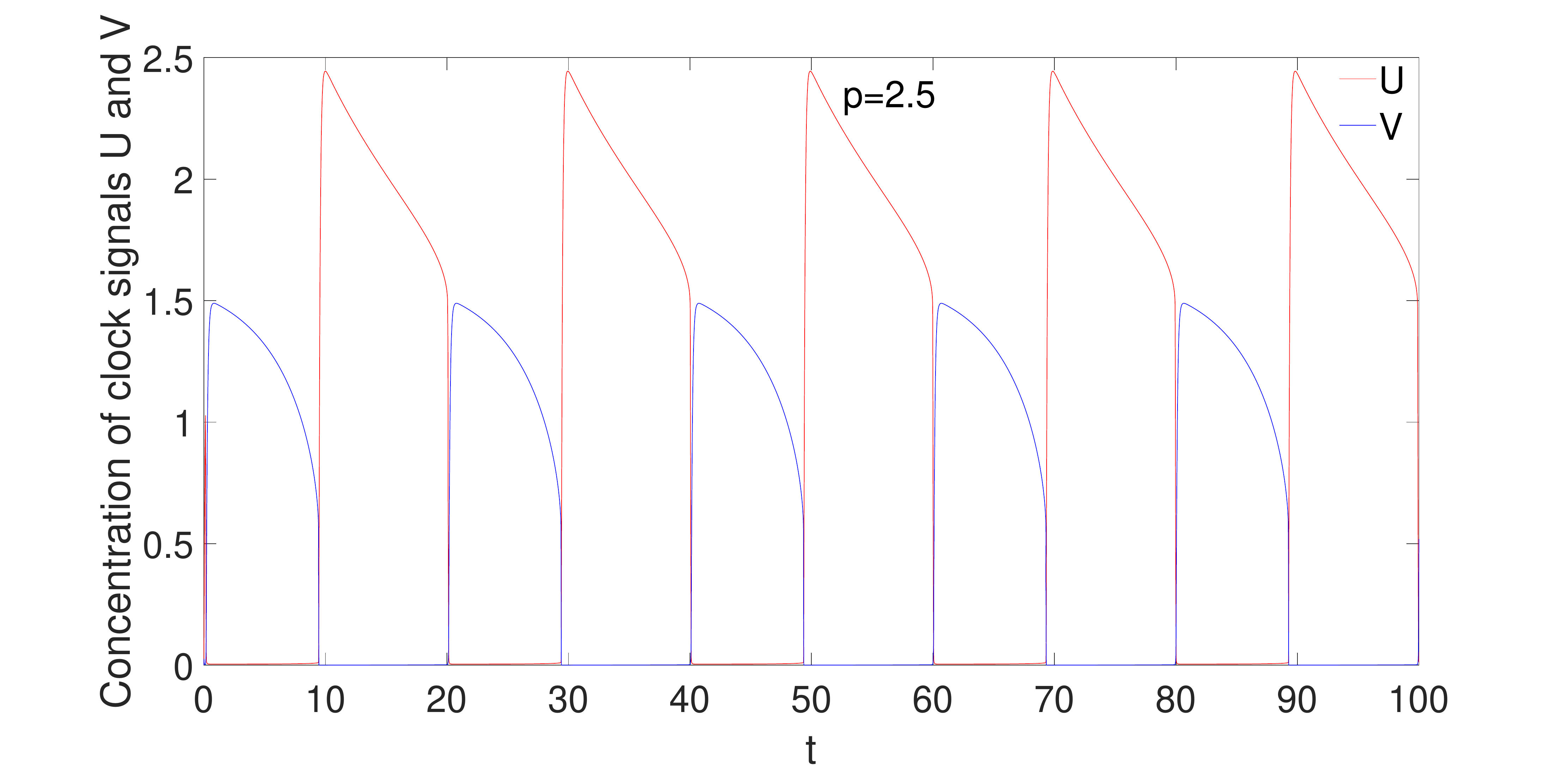}
  \hspace{2in} 
  \includegraphics[width=\columnwidth]{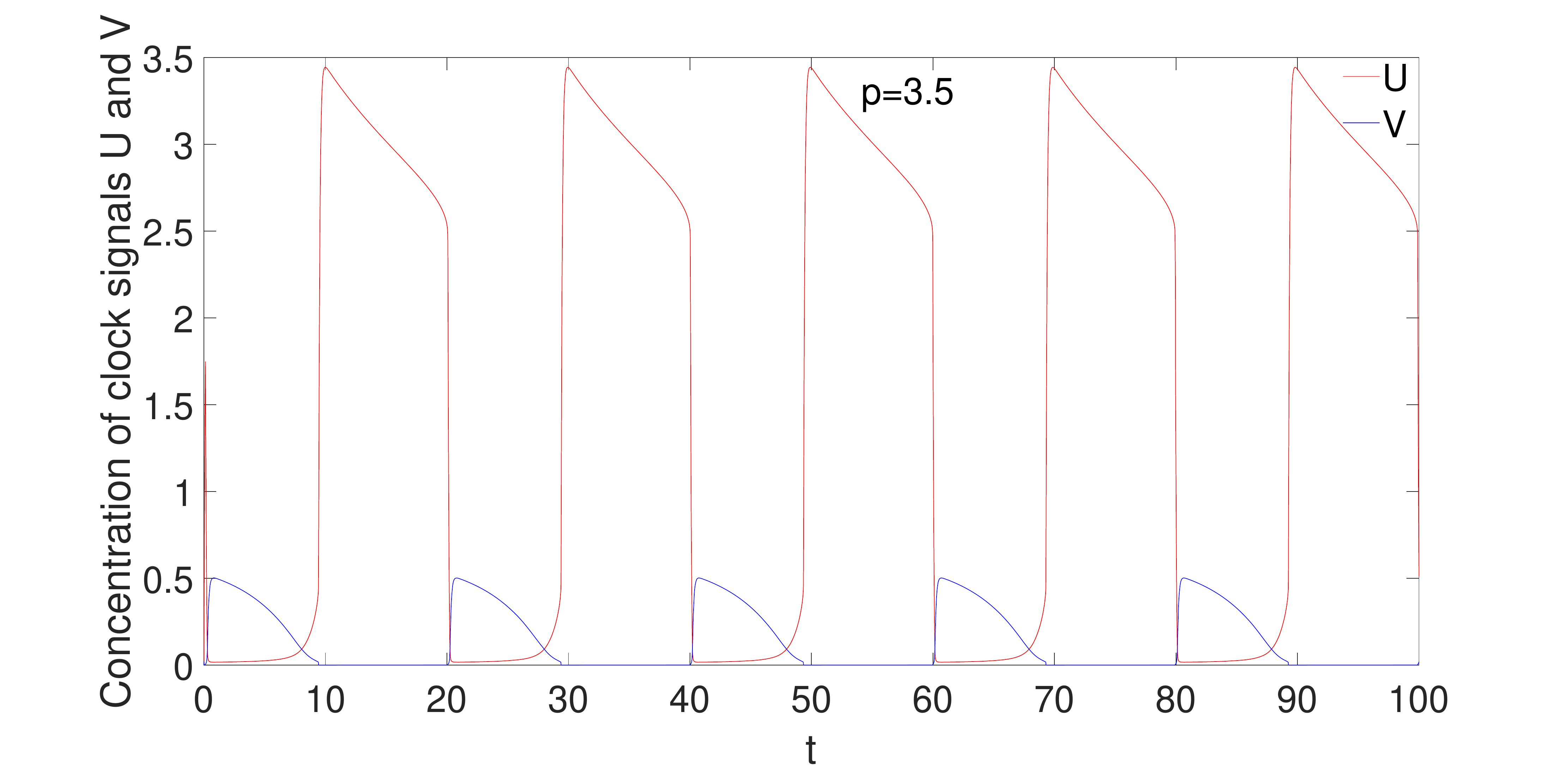}
  \caption{Simulation for $p=0.5$, $p=2.5$ and $p=3.5$.}
\end{figure}
\par Based on $p=2$, we also try different choices of parameter $c$ as 4 and 40. As we emphasize in Lemma 4.1, small value of $c$ would take $u$ and $v$ away from zero at their low amplitudes. For our symmetric clock signals, $c=400$ is enough.
\begin{figure}
  \centering
  \includegraphics[width=\columnwidth]{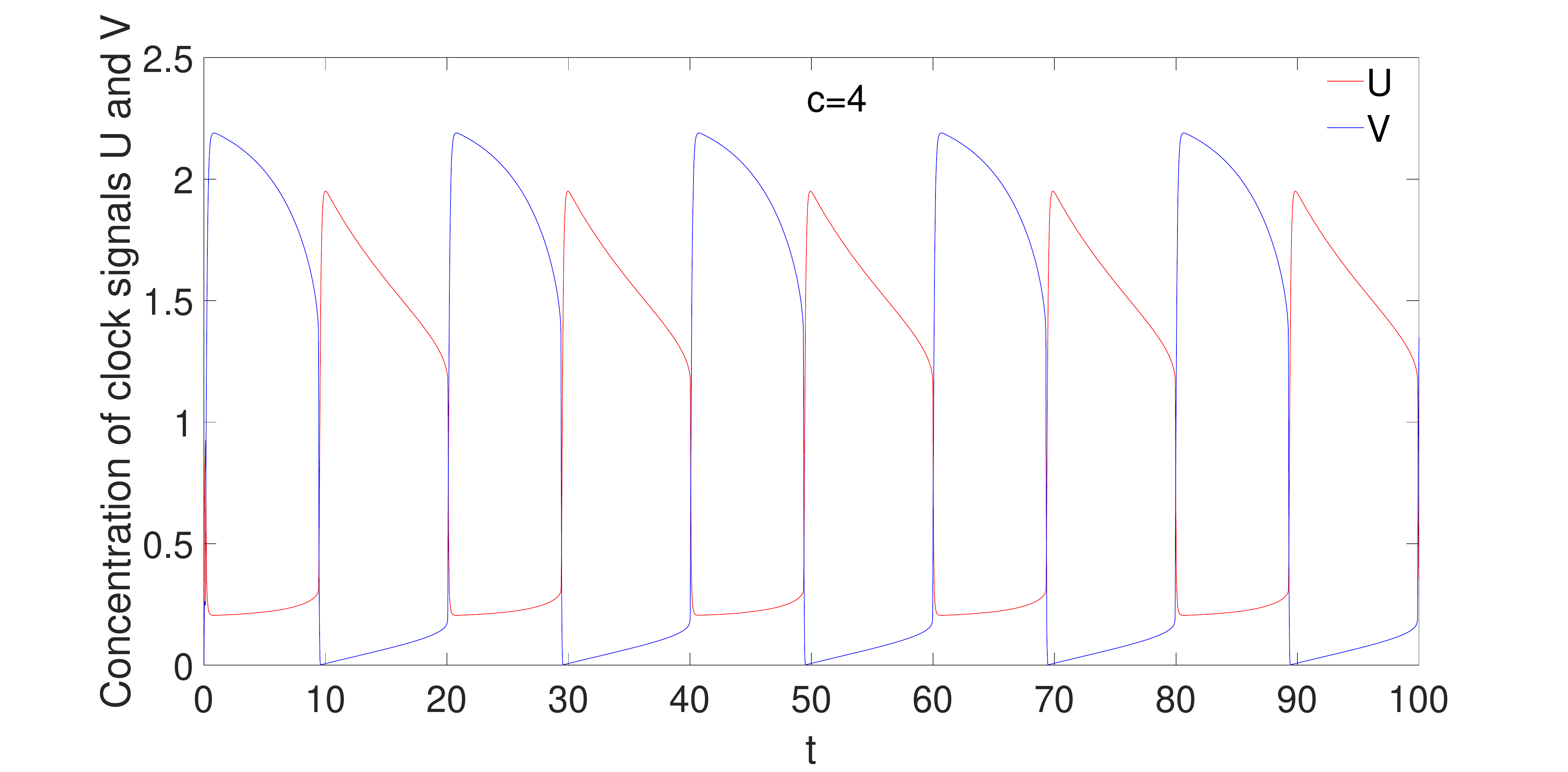}
  \hspace{2in} 
  \includegraphics[width=\columnwidth]{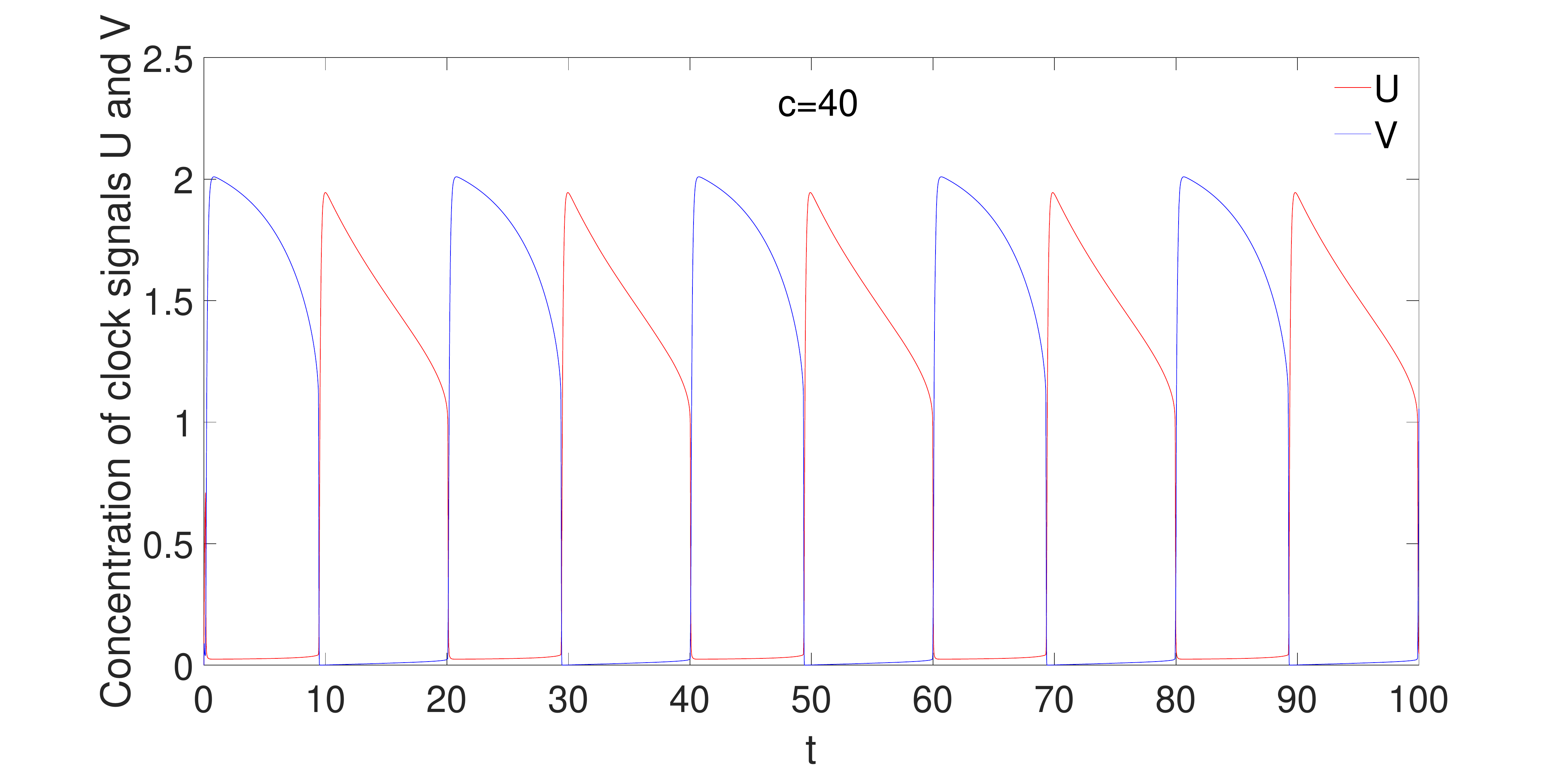}
  \caption{Simulation for $c=4$ and $c=40$.}
\end{figure}

\subsection{Parameters for coupling}
In system (9) we introduce $\eta_{1}$ and $\eta_{2}$ which we name 'parameters for coupling'. Role of these parameters is not only to adjust convergence speed and oscillation period of the isolated system, but also resulting in different timescales between these systems when coupling. 
\par Specifically, in system (9), parameter $\eta_{1}$ does not destroy structure of the critical manifold and equilibrium, so amplitude of oscillator $x$ is independent with $\eta_{1}$. However, $\eta_{1}$ affects the rate of change in $x$, which is closely related to oscillation period. We choose $\eta_{1}=0.01$ in Example 3.3 in order to magnify period of $x$ and control system $x-y$ on a slower timescale than system $u-v$. While $\eta_{2}$ adjusts period of the whole system in Example 3.3. We prefer $\eta_{2}$ to be large to speed up the convergence of $u$ and $v$, which ensures that $u$ and $v$ switch between high and low amplitudes in the form of phase mutations.
\par Till now, we conclude the choice of parameters and give a theorem that system (9) can act as a universal oscillator model to build a pair of symmetric clock signal $U$ and $V$ as we want.
\begin{theorem}
    For system (9), we choose function $f(x)$ and $x+\mu y+\lambda$ as what Theorem 3.1 states. Parameter $p$ is between the fold values of $f(x)$ i.e. $x_{m}<p<x_{M}$, c is large enough. $\eta_{1}$ is as small as possible while $\eta_{2}$ is as large as possible. Then for any initial point $(x_{0},y_{0},u_{0},v_{0})$ satisfying $x_{0}>0$, $y_{0}>0$, $u_{0}\geq 0$ and $v_{0}\geq 0$ except for the case that $(x_{0},y_{0})$ is the unique equilibrium of system $x-y$, oscillation of $u$ and $v$ would exhibit a certain symmetry i.e. corresponding species $U$ and $V$ are pair of symmetric clock signals as Definition 2.4 describes.
\end{theorem}
\begin{proof}
    As we emphasize in Theorem 3.1 and Example 3.2, isolated system $x-y$ can result in relaxation oscillation which is independent of initial point in the first quadrant of phase plane except for the equilibrium. The equilibrium of isolated system $u-v$ is shown in Lemma 4.1 that when $u=0$, $v>0$, and vice versa. Exponential convergence of the isolated system $u-v$ given by Theorem 4.1, along with small enough $\eta_{1}$ and large enough $\eta_{2}$ makes sure that $(u,v)$ can converge quite quickly to corresponding equilibrium as the value of $x$ changes. So transform system (9) back into chemical reaction networks, species $U$ and $V$ are symmetric clock signals.   $\hfill\blacksquare$
\end{proof}

Moreover, parameter choice in Theorem 4.2 almost erases the response time of $u$ and $v$ with respect to the change in $x$, so we utilize the period of $x$ to roughly estimate the period of $u$ and $v$. Imitate approach in \cite{fernandez2020symmetric}, we give a formula for calculating the period of $x$ at high amplitude and low amplitude in system (9) as Theorem 4.2.
\begin{theorem}
    Consider the relaxation oscillation orbit $\Gamma_{\epsilon}$ in system $x-y$ of system (9), the period of $x$ i.e. time it takes to travel around the closed orbit $\Gamma_{\epsilon}$ can be approximated at the first order in $\epsilon$ by $T_{1}+T_{2}+O(1)$ with
    \begin{equation}
        T_{1}= \int_{x_{A}}^{x_{m}}\frac{(f'(x)-\frac{\mathrm{d} \psi _{1}}{\mathrm{d} x}(x,\epsilon ))dx}{\eta_{1}\eta_{2}(x+\mu(f(x)-\psi _{1}(x,\epsilon ))+\lambda)(f(x)-\psi _{1}(x,\epsilon ))}\ ,
    \end{equation}
    \begin{equation}
       T_{2}= \int_{x_{C}}^{x_{M}}\frac{(f'(x)+\frac{\mathrm{d} \psi _{2}}{\mathrm{d} x}(x,\epsilon ))dx}{\eta_{1}\eta_{2}(x+\mu(f(x)+\psi _{2}(x,\epsilon ))+\lambda)(f(x)+\psi _{2}(x,\epsilon ))}\ ,
    \end{equation}
    where $\psi_{1}$ and $\psi_{2}$ are differentiable function defined separately on $(x_{A},x_{m}) \times (0,\epsilon_{0})$ and $(x_{C},x_{M}) \times (0,\epsilon_{0})$, and $\exists \xi (\epsilon) =O(\epsilon ^{2/3})$, such that 
    \begin{align}
         &\forall x \in (x_{A},x_{m}),  \left | \psi _{1}(x,\epsilon ) \right |< \xi (\epsilon )\ ,\\
&\forall x \in (x_{M},x_{C}),  \left | \psi _{2}(x,\epsilon ) \right |< \xi (\epsilon )\ ,
    \end{align}
     $\epsilon_{0}$ is small enough. 
\end{theorem}
\begin{proof}
We first confirm the formula for $T_{1}$. Lemma 3.1 declares the existence of closed orbit $\Gamma_{\epsilon}$ which is actually the trajectory of relaxation oscillation, and $\Gamma_{\epsilon}$ lies in the neighborhood of $O(\epsilon ^{2/3})$ of $\Gamma$ by Fenichel Slow Manifold Theorem. Trajectory $\Gamma$ in the non-horizontal segment is depicted by critical manifold i.e. $y=f(x)$ in phase plane. Segment of $\Gamma_{\epsilon}$ closed to ${(x,f(x)):x_{A}<x<x_{m}}$ is defined as $\Gamma_{\epsilon,l}$:
\begin{equation}
    \Gamma_{\epsilon,l}: y= \chi ^{-}(x,\epsilon )\ .
\end{equation}
So for small enough $\epsilon_{0}$, $\forall \epsilon \in (0,\epsilon_{0})$, $\exists$
\begin{equation}
    \psi_{1}: (x,\epsilon) \to f(x)-\chi ^{-}(x,\epsilon )\ .
\end{equation}
It is obvious that $\psi_{1}>0$ and is differentiable w.r.t $x$, and $\left | \psi _{1}(x,\epsilon ) \right |<O(\epsilon ^{2/3})$ for $\forall x \in (x_{A},x_{m})$. Then we can substitute $y=f(x)-\psi_{1}$ into $\frac{\mathrm{d} y}{\mathrm{d} t}$ in system (9) and get:
\begin{equation}\begin{aligned}
    \frac{\mathrm{d} y}{\mathrm{d} x}&= f'(x)-\frac{\mathrm{d} \psi _{1}}{\mathrm{d} x}\ , \\
\frac{\mathrm{d} y}{\mathrm{d} t}&=\eta_{1}\eta_{2}(x+\mu(f(x)-\psi _{1}(x,\epsilon ))+\lambda)(f(x)-\psi _{1}(x,\epsilon ))\ .
\end{aligned}\end{equation}
The time it takes to travel along $\Gamma_{\epsilon,l}$ is given by
\begin{equation}\begin{aligned}
    &T_{1}=\int _{\Gamma_{\epsilon ,l}}dt=\int_{x_{A}}^{x_{m}}\frac{1}{\frac{\mathrm{d} x}{\mathrm{d} t}}dx=\int_{x_{A}}^{x_{m}}\frac{\frac{\mathrm{d} y}{\mathrm{d} x}}{\frac{\mathrm{d} y}{\mathrm{d} t}}dx \\
    &=\int_{x_{A}}^{x_{m}}\frac{(f'(x)-\frac{\mathrm{d} \psi _{1}}{\mathrm{d} x}(x,\epsilon ))dx}{\eta_{1}\eta_{2}(x+\mu(f(x)-\psi _{1}(x,\epsilon ))+\lambda)(f(x)-\psi _{1}(x,\epsilon ))}\ .
\end{aligned}\end{equation}
Similar for $T_{2}$ i.e. time it takes to travel along $\Gamma_{\epsilon,r}$. Since $\Gamma_{\epsilon}$ in the horizontal segment corresponds to the phase mutation of $x$ between high and low amplitudes, a whole period of $x$ can be approximated as $T_{1}+T_{2}+O(1)$.                        $\hfill\blacksquare$
\end{proof}
Apply Theorem 4.2 to our Example 3.3, we get estimated period of high amplitude and low amplitude of $x$ with parameters stated before as
\begin{align*}
    T_{1}=&\int_{0}^{1}\frac{10(-3x^{2}+12x-9)}{(x-2)(-x^{3}+6x^{2}-9x+5)}dx \approx 10.47\ ,\\
    T_{2}=&\int_{4}^{3}\frac{10(-3x^{2}+12x-9)}{(x-2)(-x^{3}+6x^{2}-9x+5)}dx \approx 9.19\ .
\end{align*}
As claimed before, we directly utilize $T_{1}$ and $T_{2}$ to estimate the period of $u$ and $v$ at high or low amplitude, which is roughly consistent with the simulation result in Fig.5.
~\\
\par While in system (11), we actually use our clock signals $U$ and $V$ to realize Example 2.3 and insert a new parameter $\eta_{3}$ into the coupled ODEs. The role of $\eta_{3}$ is to coordinate the timescale between the system (10) generating the clock signals and the reaction modules to be adjusted. This is necessary because there is usually a difference between the period of the clock signal i.e. time given for reaction module to converge to equilibrium and the actual time required for reaction module to converge to equilibrium. We choose $\eta_{3}=0.35$ in this example. While we apply the universal oscillator model (9) to other reaction modules to be ordered, value of $\eta_{3}$ is depend on situation.

\section{Spontaneous termination of loop}
Plenty of studies on chemical oscillators, including previous sections of this paper, mainly focus on implementing the sequential operation of reaction modules, while how to make the system spontaneously terminate the alternate operation of two reaction modules according to certain judgment conditions is also a thought-provoking problem. As we known, oscillation can not end spontaneously, otherwise it would not be an oscillation. So we need to set up additional species to interfere with the loop of reaction modules.
\par Refer to the computer for setting instructions to jump out of a loop, there are two main methods: 
\begin{enumerate}
    \item One or more variables reaches a specific value;
    \item The loop operates for a preset number of times.
\end{enumerate}
The form asks for design according to specific situation, while the latter can be treated in a general way. Come back to our Timer Model as Example 2.3, each time the two modified reaction modules loop, concentration of species $X_{1}$ goes up by one(see Fig.6), and concentration of $X_{1}$ does not change until the modified reaction module 2 operates again. So the concentration of species $X_{1}$ i.e. value of $x_{1}$ at any given moment corresponds to the number of loops of these two reaction modules, and by slowing down the frequency of the clock signal $U$ and $V$, we can use the oscillation period of $u$ and $v$ to refer to the time for one loop of the two reaction modules to be adjusted. When we face two new reaction modules, we add clock signal $U$ and $V$ separately into these modules as catalyst and monitor the number of loops by value of $x_{1}$. For example, when $x_{1}$ stabilizes at 100 and is about to jump to 101, the two new reaction modules loop for exactly 100 times.
\par Next problem is how can the whole system make a spontaneous decision whether to end the loop based on value of $x_{1}$. Our thought also comes from adding catalysts to turn the reaction module on or off. Assume that we need the new reaction modules to loop for $n$ times, then we can build an additional species $X_{3}$ which acts as another catalyst of both the two new reaction modules. We just need concentration of $X_{3}$ to go to zero when times of loop i.e. value of $x_{1}$ increases beyond the preset number $n$, and this can be realized by a truncated subtraction module as follows:
\begin{equation}\begin{aligned}
   N+X_{3} &\to N+2X_{3}\ , \\
   X_{1}+X_{3} &\to X_{1}\ , \\
   2X_{3} &\to X_{3}\ ,
\end{aligned}\end{equation}
whose ODE expressed as:
\begin{equation}
    \frac{\mathrm{d} x_{3}}{\mathrm{d} t}=(n-x_{1}-x_{3})x_{3}\ .
\end{equation}
\par We give up the Sub module mentioned in \cite{vasic2020} for that the module produces an extra useless species $H$ and exponential convergence of it has not been proven. While in our module (23), exponential convergence is clear:
\begin{equation}
    x_{3}= \frac{n-x_{1}}{1+ke^{-(n-x_{1})t}}, k \in \mathbb{R}\ .
\end{equation}
So value of $x_{3}$ converges to equilibrium at exponential speed and the equilibrium is:
\begin{equation}
x_{3} = \begin{cases}
n-x_{1}, & \text{ if $n > x_{1}$}  \\
0, & \text{ otherwise } 
\end{cases}
\end{equation}
Note that if the initial value of $x_{3}$ is zero, then it will stay at zero forever. So when we use this module, we can just choose initial value of $x_{3}$ as $n$ along with initial value of $x_{1}$ equal to zero. Before $x_{1}$ goes beyond $n$, value of $x_{3}$ will converge to $n-x_{1}$ during each loop, and catalyst $X_{3}$ always keeps the loop going. When times of loop exceeds the preset number $n$, $x_{3}$ converges quickly to zero, turning both of the two reaction modules off. Couple the ODE (24) with ODEs (11) and keep selection of parameters and initial point unchanged, we get the simulation of $x_{3}$ compared with $x_{1}$ in Fig.9. We also provide a $x_{4}$ as subtraction between the same $n$ and $x_{2}$.
\begin{figure}
  \centering
  \includegraphics[width=\columnwidth]{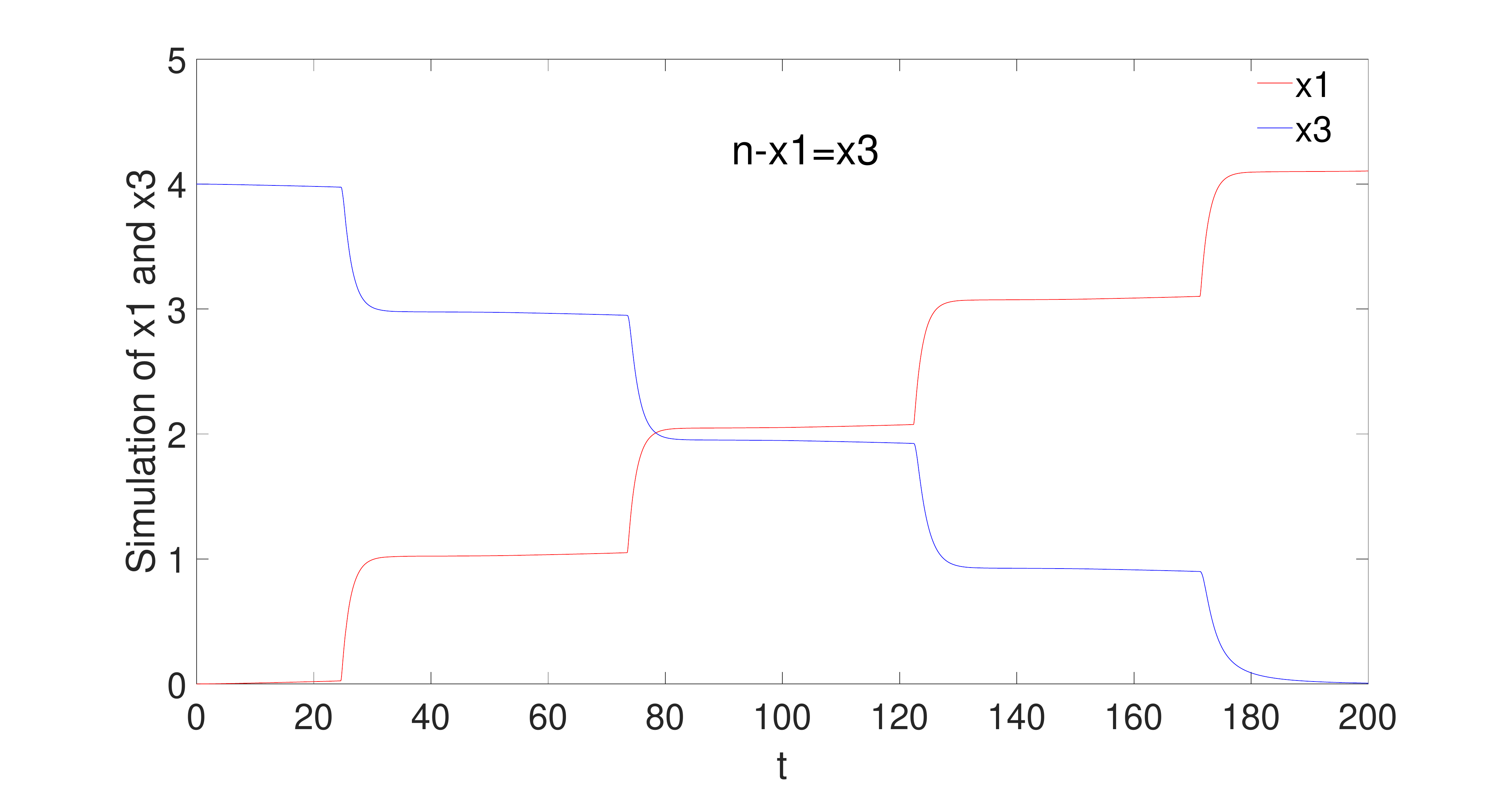}
  \hspace{2in} 
  \includegraphics[width=\columnwidth]{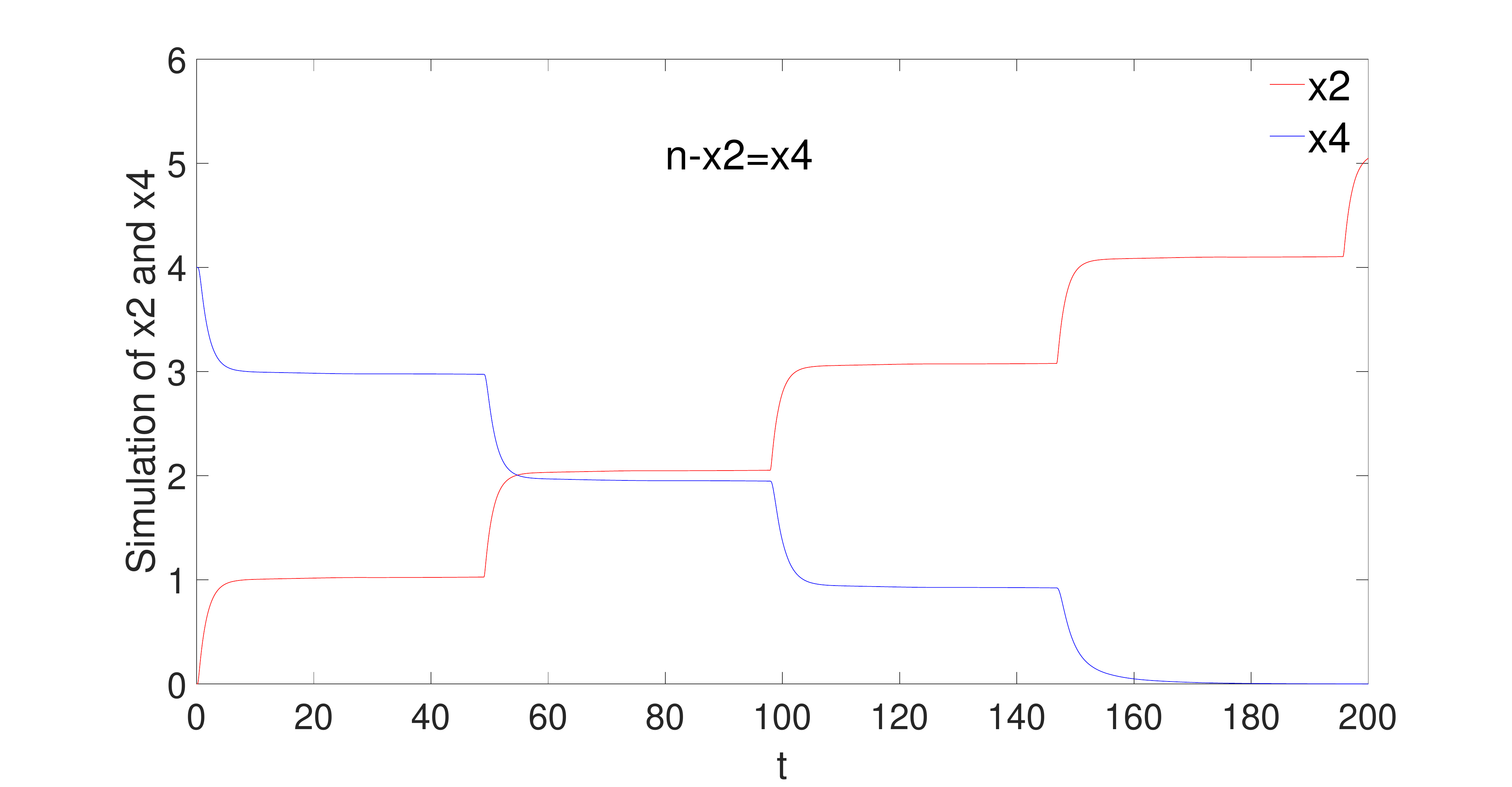}
  \caption{Simulation results of difference operation under dynamic input, with n=4. }
\end{figure}
We can easily conclude from the simulation diagram that:
\begin{enumerate}
    \item There is also a phase difference between $x_{3}$ and $x_{4}$, which is directly resulted by the phase difference between $x_{1}$ and $x_{2}$.
    \item The smoothness of $x_{3}$ in the descending section is almost equal to the smoothness of $x_{1}$ in the ascending section, and the former is affected by both the convergence speed of modified reaction module 2 in Example 2.3 and the convergence speed of our truncated subtraction module (23). So as $x_{4}$.
    \item The horizontal segment of $x_{3}$ has a slight downward trend, which implies that the horizontal segment of $x_{1}$ is not exactly horizontal. 
\end{enumerate}
\par Just from the aim of constructing the Counter Model i.e. building a component whose value increases by one every once in a while, the selection we have given in previous section is enough within acceptable limits of error. When it comes to design of spontaneous termination of loop based on this model, these errors, which could otherwise be ignored, lead to undesirable results. Therefore, in this section, we will consider stricter parameter values. 
\par In our Counter Model, values of both $x_{1}$ and $x_{2}$ increase periodically over time only with a phase difference. While the increases of $x_{1}$ and $x_{2}$ calibrate different stages within a single loop: value of component $v$ is firstly positive, leading to the increase of $x_{2}$, then $x$ oscillates at low amplitude and value of $u$ goes strictly beyond zero, resulting in the increase of $x_{1}$. When the roles of $u$ and $v$ are reversed again, one loop is finished. Thus, increase of $x_{2}$ happens in beginning of every loop, while increase of $x_{1}$ appears in the second half. We'll terminate the loop with components $x_{3}$ and $x_{4}$ by subtracting $x_{1}$ and $x_{2}$ with respect to preset loop times $n$ respectively, and demonstrate their effects.

\par Following example is given by inserting the termination component $x_{3}$ back to Counter Model as catalyst to compare the result with Counter model without termination operation.
\begin{example}
 \begin{equation}\label{Example 5.1}
    \begin{aligned}
    \frac{\mathrm{d} x}{\mathrm{d} t} &= \eta_{1}\eta_{2}(-x^3+6x^2-9x+5-y)x/\epsilon\ , \\
    \frac{\mathrm{d} y}{\mathrm{d} t} &= \eta_{1}\eta_{2}(x-\rho)y\ , \\
    \frac{\mathrm{d} u}{\mathrm{d} t} &= \eta_{2}(p-u-cuv)\ , \\
    \frac{\mathrm{d} v}{\mathrm{d} t} &= \eta_{2}(x-v-cuv)\ , \\
    \frac{\mathrm{d} x_{1}}{\mathrm{d} t} &= \eta_{3}(x_{2}-x_{1})u\ , \\
    \frac{\mathrm{d} x_{2}}{\mathrm{d} t} &= \eta_{3}(x_{1}+1-x_{2})v\ , \\
    \frac{\mathrm{d} x_{3}}{\mathrm{d} t} &= \eta_{4}(n-x_{1}-x_{3})x_{3}\ , \\
    \frac{\mathrm{d} x_{1}'}{\mathrm{d} t} &= \eta_{5}(x_{2}'-x_{1}')ux_{3}\ , \\
    \frac{\mathrm{d} x_{2}'}{\mathrm{d} t} &= \eta_{5}(x_{1}'+1-x_{2}')vx_{3}\ . \\
    \end{aligned}
\end{equation}
We use $x_{1}'$ and $x_{2}'$ as comparison of $x_{1}$ and $x_{2}$ under termination component $x_{3}$. In order to reduce the error caused by the failure to completely turn the module off at the corresponding time because $u$ and $v$ do not reach zero at their respective low amplitudes, we increase the value of parameter $c$ to 5000. And take $\eta_{4}$ as large as possible to speed up response of $x_{3}$ towards changes of $x_{1}$. Let $\eta_{3}$ equal to $\eta_{5}$ for fairness of comparison. In this example, we choose $\epsilon=0.001$, $\rho=2$, $p=2$, $c=5000$, $n=4$, $\eta_{1}=0.01$, $\eta_{2}=10$, $\eta_{3}=\eta_{5}=1$, $\eta_{4}=500$, and get the simulation in Fig.10.  
\end{example}
\begin{figure}
  \centering
  \includegraphics[width=\columnwidth]{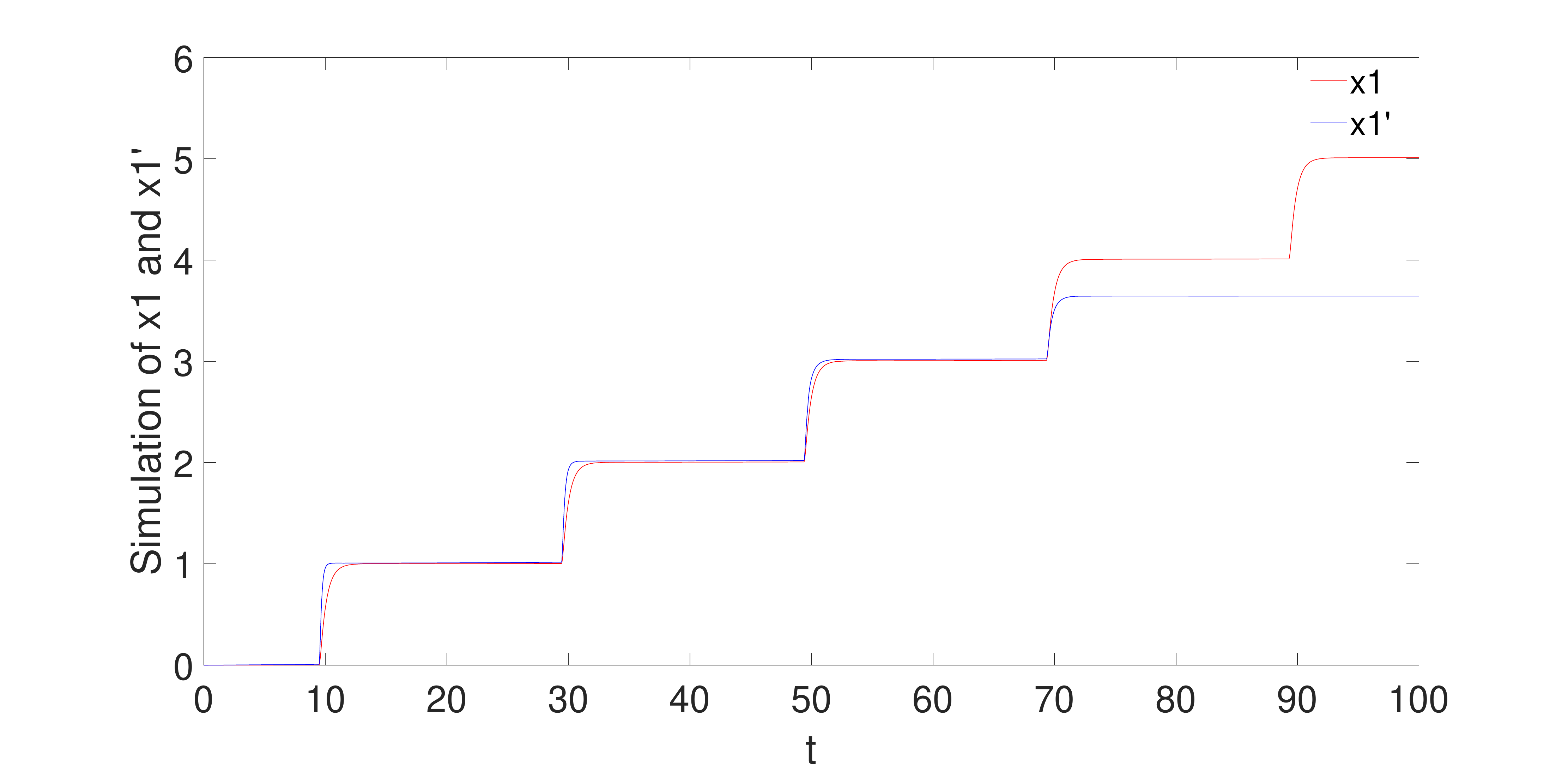}
  \hspace{2in} 
  \includegraphics[width=\columnwidth]{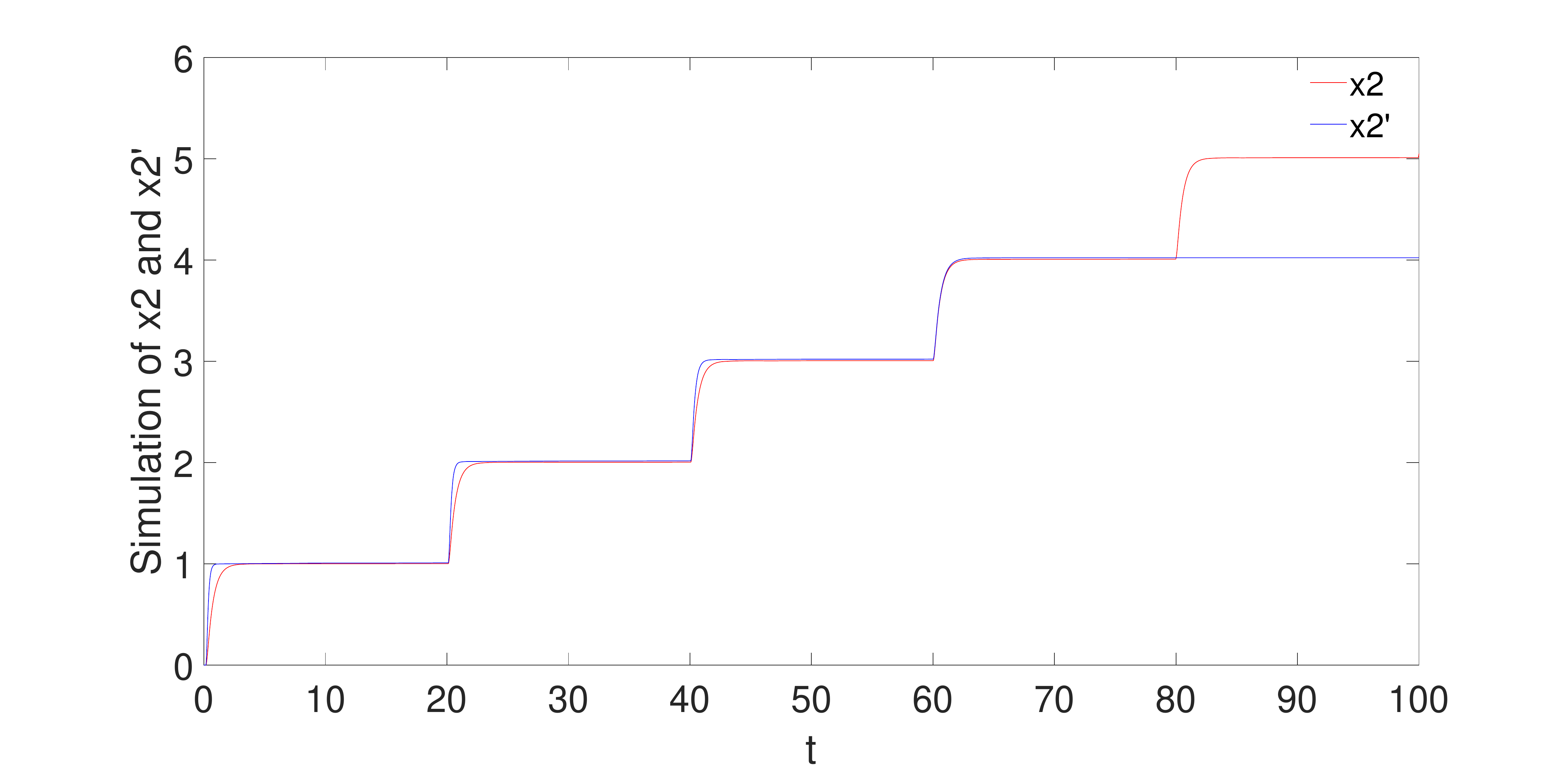}
  \caption{Comparison of $x_{1}'$ with $x_{1}$, $x_{2}'$ with $x_{2}$ under termination component $x_{3}$. }
\end{figure}
\par We also offer Example 5.2 utilizing $x_{4}$ as termination component and give corresponding simulation result under the same selection of parameters in Fig.11.
\begin{example}
  \begin{equation}
    \begin{aligned}
    \frac{\mathrm{d} x}{\mathrm{d} t} &= \eta_{1}\eta_{2}(-x^3+6x^2-9x+5-y)x/\epsilon\ , \\
    \frac{\mathrm{d} y}{\mathrm{d} t} &= \eta_{1}\eta_{2}(x-\rho)y\ , \\
    \frac{\mathrm{d} u}{\mathrm{d} t} &= \eta_{2}(p-u-cuv)\ , \\
    \frac{\mathrm{d} v}{\mathrm{d} t} &= \eta_{2}(x-v-cuv)\ , \\
    \frac{\mathrm{d} x_{1}}{\mathrm{d} t} &= \eta_{3}(x_{2}-x_{1})u\ , \\
    \frac{\mathrm{d} x_{2}}{\mathrm{d} t} &= \eta_{3}(x_{1}+1-x_{2})v\ , \\
    \frac{\mathrm{d} x_{4}}{\mathrm{d} t} &= \eta_{4}(n-x_{2}-x_{4})x_{4}\ , \\
    \frac{\mathrm{d} x_{1}'}{\mathrm{d} t} &= \eta_{5}(x_{2}'-x_{1}')ux_{4}\ , \\
    \frac{\mathrm{d} x_{2}'}{\mathrm{d} t} &= \eta_{5}(x_{1}'+1-x_{2}')vx_{4}\ . \\
    \end{aligned}
\end{equation}
\end{example}
\begin{figure}
  \centering
  \includegraphics[width=\columnwidth]{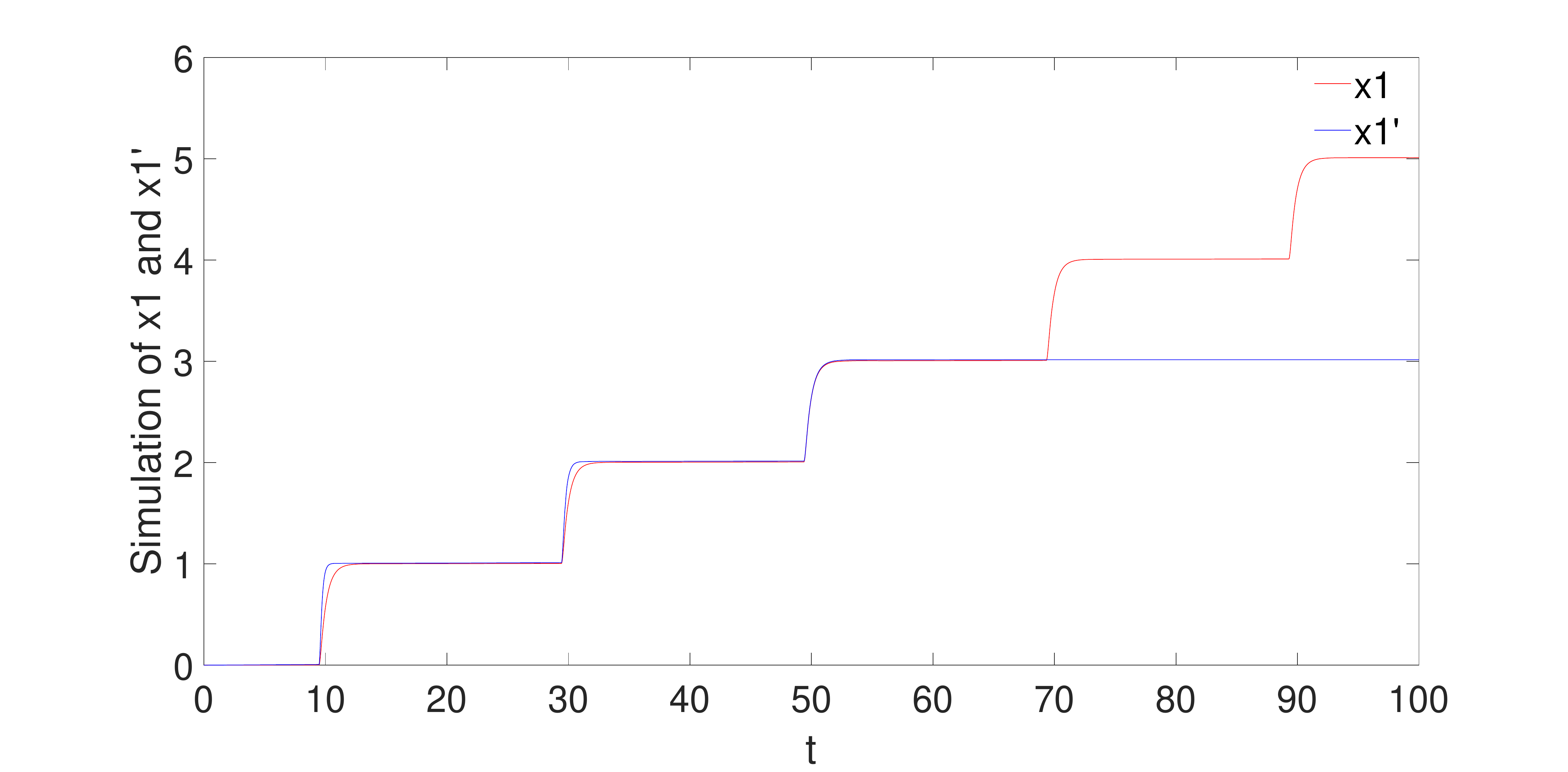}
  \hspace{2in} 
  \includegraphics[width=\columnwidth]{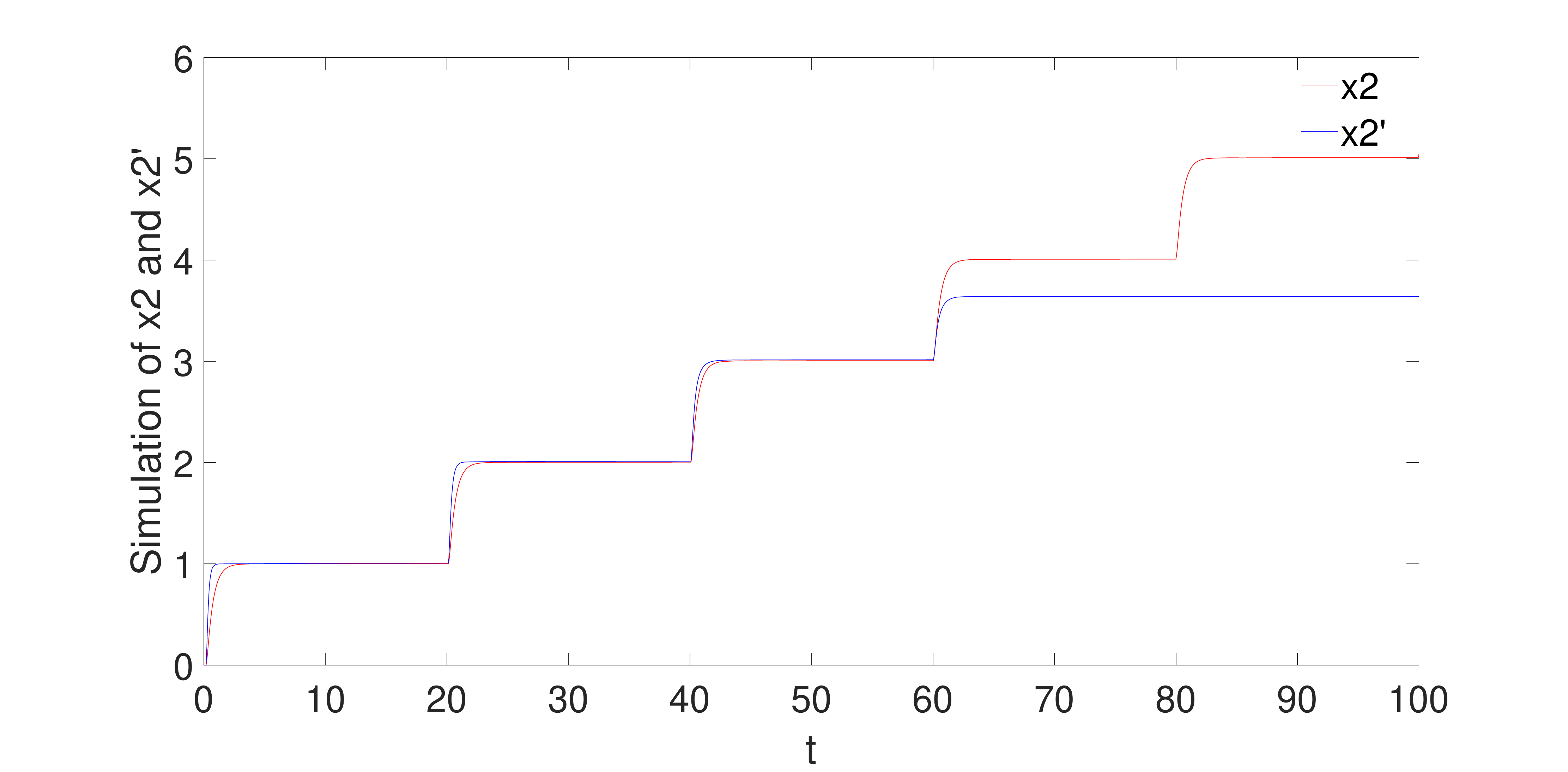}
  \caption{Comparison of $x_{1}'$ with $x_{1}$, $x_{2}'$ with $x_{2}$ under termination component $x_{4}$. }
\end{figure}
\par Different choices of termination component contribute to different simulation results of loop termination. Specifically, when using truncated subtraction of preset $n$ and $x_{1}$ i.e. $x_{3}$ as termination component, change of $x_{3}$ happens after $x_{1}$ increases by one, which is corresponding to the beginning of second half in one loop. So with preset loop times $n=4$, after the instruction $x_{1}'+=1$ operates for three times, $x_{1}'=x_{2}'=3$(See the blue lines in Fig.10), $x_{1}=x_{2}=3$(see the red lines in Fig.10) and $x_{3}=1$. Then the fourth loop happens, during the first half, $x_{2}$ and $x_{2}'$ increase to 4, while $x_{1}$ and $x_{1}'$ remain unchanged. At the beginning of the second half, $x_{1}$ converges to 4 at exponential speed, $x_{3}$ responds immediately to this change(Response speed is controlled by $\eta_{4}$) and converges to zero, which sets the rate of change in $x_{1}'$ and $x_{2}'$ to zero and terminates the loop of Timer Model based on $x_{1}'$ and $x_{2}'$. So in Fig.10, when the loop is ended, $x_{2}'$ stays at 4 while $x_{1}'$ is between 3 and 4 for that during $x_{2}$ converges to 4 and causes change in $x_{3}$, $x_{2}'$ is also increasing and the increase of $x_{2}'$ is aborted because of annihilation of $x_{3}$. In summary, termination of the loop controlled by $x_{3}$ occurs after $x_{2}$ increases to the preset $n=4$, when the fourth loop enters the second half. Mechanism in Example 5.1 terminates loop at the beginning of second half of the $n$th loop.
\par While in Example 5.2, we choose $x_{4}$ as termination component. When the third loop ends, $x_{1}=x_{2}=x_{1}'=x_{2}'=3$ and $x_{4}=1$. Next time the fourth loop begins, $x_{2}$ converges to 4 at exponential speed and quickly $x_{4}$ responds to zero, the loop of system $x_{1}'-x_{2}'$ is turned off. So mechanism in this example terminates loop at the beginning of $n$th loop, which can be seen in Fig.11.

\par Note that neither $x_{3}$ nor $x_{4}$ can exactly terminate the loop just as the last loop ends. This default can not be overcome under our mechanism for that both convergence and response take time, and the time at which our termination component ends the loop always lags behind our ideal time. Although we can weaken this lag by adjusting parameters, we emphasize that termination of loop with such precision is sometimes unnecessary in practice. In our Counter Model, if we just choose $x_{1}$ as the role of counter, every time $x_{1}$ increases when the loop enters the second half, so ending the loop at the beginning of a new one is enough. As Fig.11 shows, if we just want the loop of system $x_{1}'-x_{2}'$ to operate for $n$ times and output the termination value of $x_{1}'$, we just need to use truncated subtraction $n+1-x_{2}$ as termination component and add it into system $x_{1}'-x_{2}'$ as Example 5.2 demonstrates. Similar cases include judging first and then performing the corresponding operation, such as module comparing output with threshold and module performing weight learning in supervised neural network based on chemical reactions \cite{arredondo2022supervised}, and case where the main computation is concentrated in the second half of the loop. For these situations, our strategy of loop termination as Example 5.2 is effective. While if the main computation is concentrated in the first half of the loop, consider the corresponding strategy in Example 5.1.
\par We transform the ODEs in the Example 5.2 into the corresponding chemical reaction network using mass-action kinetics, and conclude the section with this.
\begin{align*}
    4X&\overset{\eta_{1}\eta_{2}/\epsilon }{\rightarrow}3X\ , \\
3X&\overset{6\eta_{1}\eta_{2}/\epsilon }{\rightarrow}4X\ , \\
2X&\overset{9\eta_{1}\eta_{2}/\epsilon }{\rightarrow}X\ , \\
X&\overset{5\eta_{1}\eta_{2}/\epsilon }{\rightarrow}2X\ , \\
X+Y&\overset{1\eta_{1}\eta_{2}/\epsilon }{\rightarrow}Y\ , \\
X+Y&\overset{\eta_{1}\eta_{2}}{\rightarrow}2Y\ , \\
2Y&\overset{\eta_{1}\eta_{2}\rho}{\rightarrow}Y\ , \\
\varnothing &\overset{\eta_{2}p}{\rightarrow}U\ , \\
U &\overset{\eta_{2}}{\rightarrow}\varnothing\ , \\
U+V &\overset{\eta_{2}c}{\rightarrow}V\ , \\
X &\overset{\eta_{2}}{\rightarrow}V\ , \\
V&\overset{\eta_{2}}{\rightarrow}\varnothing\ , \\
U+V &\overset{\eta_{2}c}{\rightarrow}U\ , \\
X_{2}+U &\overset{\eta_{3}}{\rightarrow}X_{1}+X_{2}+U\ , \\
X_{1}+U &\overset{\eta_{3}}{\rightarrow}U\ , \\
X_{1}+V &\overset{\eta_{3}}{\rightarrow}X_{1}+X_{2}+V\ , \\
V &\overset{\eta_{3}}{\rightarrow}X_{2}+V\ , \\
X_{2}+V &\overset{\eta_{3}}{\rightarrow}V\ , \\
N+X_{4} &\overset{\eta_{4}}{\rightarrow}N+2X_{4}\ , \\
X_{2}+X_{4} &\overset{\eta_{4}}{\rightarrow}X_{2}\ , \\
2X_{4} &\overset{\eta_{4}}{\rightarrow}X_{4}\ , \\
X_{2}'+U+X_{4} &\overset{\eta_{5}}{\rightarrow}X_{1}'+X_{2}'+U+X_{4}\ , \\
X_{1}'+U+X_{4} &\overset{\eta_{5}}{\rightarrow}U+X_{4}\ , \\
X_{1}'+V+X_{4} &\overset{\eta_{5}}{\rightarrow}X_{1}'+X_{2}'+V+X_{4}\ , \\
V+X_{4} &\overset{\eta_{5}}{\rightarrow}X_{2}'+V+X_{4}\ , \\
X_{2}'+V+X_{4} &\overset{\eta_{5}}{\rightarrow}V+X_{4}\ . 
\end{align*}

\section{General Process of Placing Oscillator Components into Reaction Modules}
Our primary goal in designing chemical oscillator is to achieve efficient molecular computation. Recent attempts to build artificial neural networks in biochemical environments \cite{moorman2019dynamical,vasic2021programming,chiang2015reconfigurable,blount2017feedforward,arredondo2022supervised,anderson2021reaction} have not only improved the computational power of molecular computers \cite{benenson2004autonomous}, but also helped advance the understanding of how living cells perform complex operations. 
\par There have been many ways to build supervised chemical neural network such as multilayer perceptron model and recurrent neural network, with difference lying in the selection of kinetics and chemical reaction network to realize each step of operation instruction. Most chemical neural networks need to adjust the operation sequence of modules in the process of feed-forward value transmission(For example, the reaction module of the later layer needs to wait for the previous layer to complete the operation before performing the corresponding operation), which can theoretically be solved by setting up multiple sets of oscillators. However, work of Vasic et al. \cite{vasic2021programming} on non-competitive CRNs showed that when selecting a specific chemical reaction network structure, the execution order between different modules in the process of feed-forward value transmission does not affect the results of the output layer. In other words, problem about module execution order, which can be avoided by selection of chemical reactions, is not worth the trouble of designing oscillators. While the problem of setting operation sequence for the feed-forward value transmission module and the weight learning module using back propagation algorithm cannot be avoided, because reactions in these two modules tend to share the same species as reactants, which violates the prerequisite of non-competitive CRNs \cite{vasic2021programming}. In Fig.12, we abstract these two reaction modules separately as feed-forward module and back propagation module, and demonstrate the adjustment of oscillator components to the corresponding reaction module.
 \begin{figure}[!t]
 	\centerline{\includegraphics[width=\columnwidth]{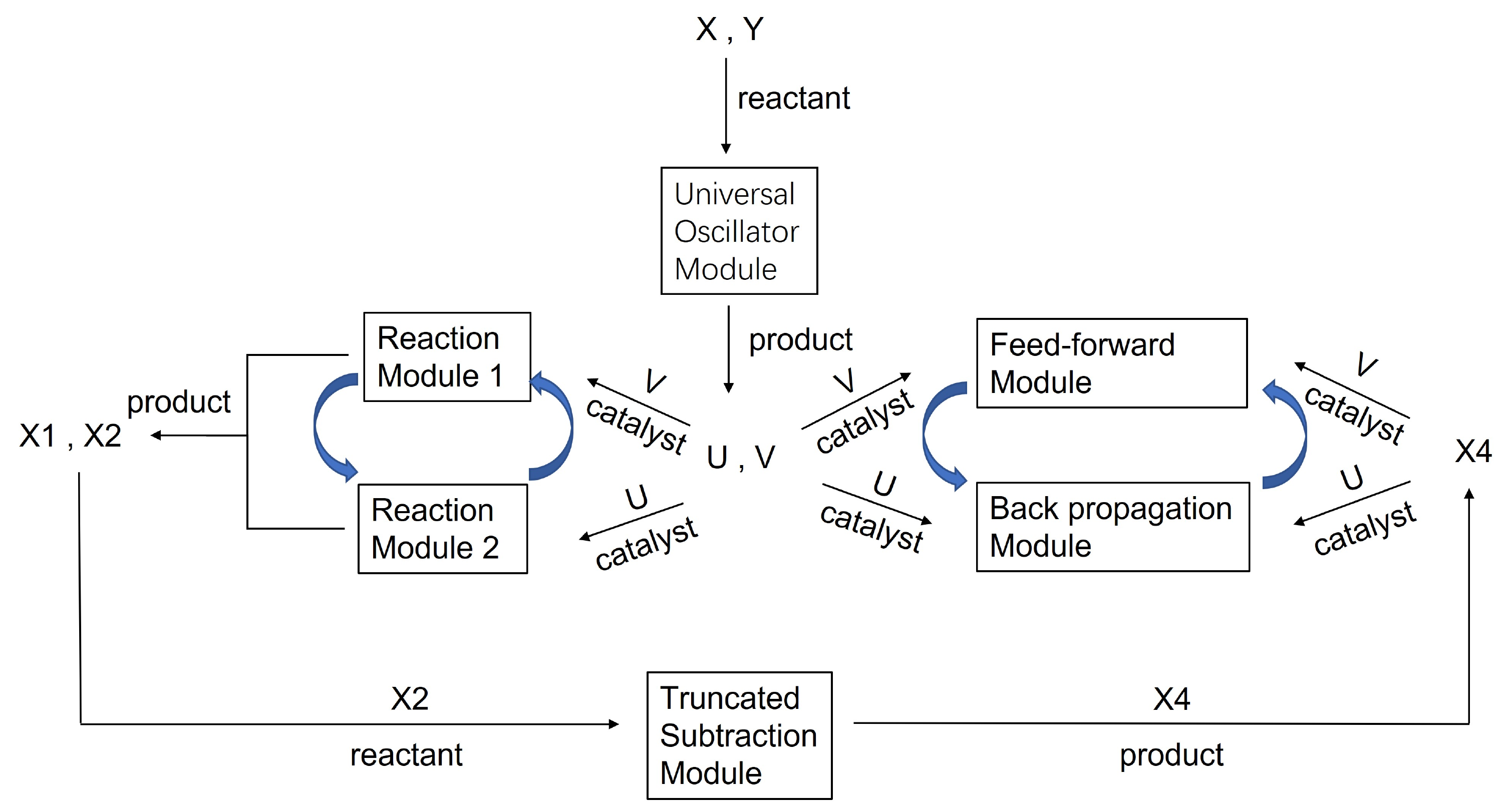}}
 	\caption{Flow chart of placing oscillator components into reaction modules}.
 	\label{fig12}
 \end{figure}
\par We use a similar approach to Example 5.2 to achieve the termination species $X_{4}$. As we emphasize in section \uppercase\expandafter{\romannumeral5}, such a design would make the chemical neural network turn off at the beginning of the $n+1$th feed-forward value transmission process after a preset number $n$ loops between the feed-forward process and back propagation process. The feed-forward process does not change the weight values, so the lag of our model for loop termination is irrelevant for training supervised chemical neural network.
\par Note that only $X$ and $Y$ of the species involved in our oscillator model require strictly given non-zero initial concentration, and the initial concentration of $X$ and $Y$ determines the initial phase of $U$ and $V$. In Example 3.2, we select specific expression of relaxation oscillation, and let the corresponding initial point be $(x,y)=(1,1)$. This ensures that $V$ enters the non-zero phase before $U$(in Fig.5), so we put clock signal $V$ into the module that needs to be prioritized i.e. Feed-forward Module in Fig.12. We can choose the initial concentration of $X$ and $Y$ flexibly according to actual needs, which is also the advantage of our oscillator model. Meanwhile, we prefer initial concentration of $X_{4}$ equal to the preset number $n$ to avoid unnecessary error in loop termination.
\par Although we demonstrate the process of placing oscillator components into the modules with form of flow chart in Fig.12, all of the reactions involved are not constrained by artificial segregation. After we set the initial concentration for all of the species in the system, reaction modules operate in turn due to the concentration change of oscillator components, rather than human intervention. Such design helps simulate more autonomous molecular computation.

\section{Conclusion and Discussions}
In this paper we develop a systematic approach to realize synchronous sequential computation with abstract chemical reactions. Our ultimate goal is to embed complex calculations into biochemical environments, and after setting the initial values of species and reaction rates, the biochemical system could run automatically to complete the target calculation task. We set up a universal chemical oscillator structure to solve the problem of how to stagger the previously disorderly reactions to make these reactions happen in the order we want. Much of the previous work mentioning chemical oscillators followed the logic of usability, how the oscillation is generated, how it is controlled by the model parameters and how the setting of initial values affects the oscillation properties are not involved in these work \cite{arredondo2022supervised,vasic2020,jiang2011}. While theoretical analysis of the models and mechanisms that cause oscillations is improving day by day \cite{tyson1980target,krupa2001relaxation,fernandez2020symmetric}, it is feasible in theory to design transparent chemical oscillators according to actual needs. Inspired by this, we give a universal approach of designing chemical oscillators to control the sequence of two reaction modules.
\par Different from the harmonic oscillators used in previous work \cite{arredondo2022supervised,jiang2011}, we choose relaxation oscillation as underlying structure of our oscillator model for that mechanism of relaxation oscillation is clear and it is robust independent to initial points. While existence and property of harmonic oscillators are depend on the selection of initial point, which is not flexible in response to specific application requirements. Besides, harsh selection of initial points often causes difficulties for biochemical implementation. In our design, parameters and structure of the chemical oscillator can be dynamically adjusted according to the needs of actual use, and to some extent, it is our oscillator model that adapts to the actual needs, rather than the other way around.
\par We explore the steps of building oscillator model and generating pair of symmetric clock signals, and give a simple example (Example 3.3) to fulfill our aim. As far as we know, to get a pair of symmetric clock signals, at the level of chemical reaction network theory, requires at least four species. In \cite{jiang2011}, the authors constructed a reaction network involved 12 species. In \cite{arredondo2022supervised}, to get two clock signals, dimension of the oscillator model is four. Therefore, our oscillator model is concise enough in terms of the number of species used. Selection of $f(x,y)$ in ODEs (4) is also flexible, taking it as a simple cubic function is enough for the rest of design. Tyson and Fife \cite{tyson1980target} abstracted another expression of $f(x,y)$($f(x,y)=x(1-x)-by(x-a)/(x+a)$, while $g(x,y)=x-y$, $a$ and $b$ are parameters) according to real chemical reactions. Substituting this set of structure, our oscillator model is still usable.
\par Although our analysis of the model mainly focus on the ODE level, we still fully consider its correspondence with chemical reactions when building the ODE model. Only when we set up the model for triggering relaxation oscillation, we do not specifically select ODE structures directly related to chemical reactions because models abstracted from biochemical examples are often too complex for theoretical analysis. Designs of other steps are derived directly from abstract chemical reactions, based on the principle that deterministic chemical kinetics is Turing universal \cite{fages2017strong}. We turn a whole example (Example 5.2) into abstract chemical reaction network at the end of section \uppercase\expandafter{\romannumeral5}, embedding additional parameters we introduce in rate constants of the corresponding reactions. Our analysis ends with the ODE simulation and the corresponding abstract chemical reaction network, while the subsequent work such as transform these abstract reactions into chemistry, can be achieved by DNA strand displacement cascades, which is beyond the scope of this paper.
\par We tested the effect of our oscillator model under the Counter Model, which can basically achieve the purpose we want. While faced with more complicated task, such as designing reaction sequences for the modules of a complete biochemical feed-forward neural network, Our oscillator models actually function as hubs:  combining and splicing the reaction modules to achieve a complete operation. Although this task is implemented in the same way that we realize the simple instruction $x_{1}+=1$, implementation of a whole biochemical feed-forward neural network is much more difficult and consists of large number of parameters to be analyzed. In the future we will do further analysis on implementing such more complex calculations.
\par Note that selection of parameters could arise dynamic behaviour as we want, our model is not as accurate as it looks, yet. We choose parameter $c$ and $\eta_{4}$ as large as possible in order to turn the low amplitude of $u$ and $v$ close enough to zero and accelerate the convergence speed. Limited by what we know about oscillations, we can only do so much. Other work such as \cite{arredondo2022supervised} is also a similar process to make the result look perfect. We think it's more reasonable to apply those constraints to reaction rate constants than to pick harshly selected initial values of species. 
\par Different from previous perspectives, we believe that how to make the system spontaneously terminate loops controlled by clock signals is also an important topic, and we give a feasible method to tackle with this that works in some situations. While as we emphasize in section \uppercase\expandafter{\romannumeral5}, our idea of ending loops is not universal for that our design does not allow the whole system to spontaneously turn off the entire loop at the end of the $n$th loop. This is because the termination component itself needs to respond to the new loop and then close it, and the resulting lag cannot be overcome by the model itself. We've tried other approaches, such as setting a module that performs the Sigmoid function to dynamically set the termination component to zero or one, which still fail to overcome the lag and increase the complexity of model. How to make our loop termination strategy more efficient is also a problem for future.
\par Our work provides theoretical analysis and assurance for embedding efficient algorithms in fields such as machine learning into biochemical environments, and oscillation plays an indispensable role in it. Different from modeling and analyzing the oscillations observed in biochemical experiments, it is also an attractive research content to design models to achieve the desired functions based on the understanding of oscillation. Recently, there has also been some work to build machine learning algorithms using oscillations. In \cite{rusch2020coupled,rusch2022graph}, the authors designed new structure of recurrent neural network and graph neural network based on coupled oscillators. how oscillations and the knowledge within the field of dynamical systems associated with oscillations, can serve other fields such as molecular computing and machine learning, will also be the focus of our future research.


\bibliographystyle{model1-num-names}

\bibliography{cas-refs}

\bio{}
\endbio

\end{document}